\documentclass[reqno]{amsart}

\usepackage{graphics}
\usepackage{color}
\usepackage{cancel}
\usepackage{tikz}
\usetikzlibrary{matrix}
\usetikzlibrary{arrows}
\usepackage{amsmath}
\usepackage{amsthm}
\usepackage{amssymb}
\usepackage{verbatim}
\usepackage{setspace}
\usepackage{mathrsfs} 
\usepackage{enumitem} 

\newcommand{\N}{\mathbb{N}}
\newcommand{\V}{\mathbf{V}}

\newcommand{\embeds}{\hookrightarrow}
\newcommand{\onto}{\twoheadrightarrow}

\renewcommand{\hat}{\widehat}

\newcommand{\liff}{\leftrightarrow}
\newcommand{\dual}{\partial}
\newcommand{\isom}{\cong}

\newcommand{\Q}{\mathbb{Q}}
\newcommand{\R}{\mathbb{R}}
\renewcommand{\leq}{\leqslant}
\renewcommand{\geq}{\geqslant}

\theoremstyle{plain}
\newtheorem{thm}{Theorem}

\newtheorem*{dualitythm}{Stone-Priestley Duality Theorem}
\newtheorem{lem}[thm]{Lemma}
\newtheorem{cor}[thm]{Corollary}

\newtheorem{prop}[thm]{Proposition}

\theoremstyle{definition}
\newtheorem{dfn}[thm]{Definition}
\newtheorem{exa}[thm]{Example}

\newtheorem{rem}[thm]{Remark}

\newtheorem*{notation}{Notation}

\numberwithin{thm}{section}

\newcommand{\lincol}{black}
\newcommand{\linth}{thick}
\newcommand{\po}[2][\pocol]{\filldraw[#1](#2) circle (2 pt);}
\newcommand{\li}[1]{\draw[\linth,\lincol] #1;}

\let\oldmarginpar\marginpar
\renewcommand\marginpar[1]{\-\oldmarginpar[\raggedleft\footnotesize #1]{\raggedright\footnotesize #1}}

\newcounter{mnotes}
\newcommand{\iplus}{\, \overline{\oplus} \,}
\newcommand{\iminus}{\, \overline{\ominus} \,}

\newcommand{\dom}{\mathrm{dom}}

\DeclareMathOperator{\C}{\rm C}

\usepackage{hyperref}

\title[Sheaves for MV-algebras and abelian $\ell$-groups via duality]{Sheaf representations of MV-algebras and lattice-ordered abelian groups via duality}
\author[M. Gehrke]{Mai Gehrke}
\address[M. Gehrke]{LIAFA, CNRS and Universit\'e Paris Diderot, Paris 7\\Case 7014\\F-75205 Paris Cedex 13\\France}
\email{mgehrke@liafa.univ-paris-diderot.fr}

\author[S. J. van Gool]{Samuel J. van Gool}
\address[S. J. van Gool]{Universit\"at Bern \\ Mathematisches Institut \\ Sidlerstrasse 5 \\ 3012 Bern \\ Switzerland }
\email{samvangool@me.com}

\author[V. Marra]{Vincenzo Marra}
\address[V. Marra]{Universit\`a degli Studi di Milano\\Dipartimento di Matematica ``Federigo Enriques''\\via Cesare Saldini, 50\\20133 Milano, Italy}
\email{vincenzo.marra@unimi.it}
\subjclass[2010]{Primary: 06D35. Secondary: 06F20; 06D50; 18F20; 54B40.}
\keywords{MV-algebra, lattice-ordered abelian group, distributive lattice, Stone duality, Priestley duality, sheaf representation, \'etale decomposition, spectral space, compact Hausdorff space, stably compact space.}

\begin{document}
\hspace{1mm}
\vspace{-5mm}

\maketitle

\vspace{-9mm}

\begin{abstract}
We study representations of MV-algebras --- equivalently,  unital lattice-ordered a\-be\-lian groups --- through the lens of Stone-Priestley duality, using canonical extensions as an essential tool. Specifically, the theory of canonical extensions implies that the (Stone-Priestley) dual spaces of MV-algebras carry the structure of  topological partial commutative ordered semigroups. We use this structure to obtain two different decompositions of such spaces, one indexed over the prime MV-spectrum, the other over the maximal MV-spectrum. These decompositions yield sheaf representations of MV-algebras, 
using a new and purely duality-theoretic result that relates certain sheaf representations of distributive lattices to  decompositions of their dual spaces.
Importantly, the proofs of the MV-algebraic representation theorems that we obtain in this way are distinguished from the existing work on this topic by the following features: (1) we use only basic algebraic facts about MV-algebras; (2) we show that the two aforementioned sheaf representations are special cases of a common result, with potential for generalizations; and (3) we show that these results are strongly related to the structure of the Stone-Priestley duals of MV-algebras. 
In addition, using our analysis of these decompositions, we prove that MV-algebras with isomorphic underlying lattices have homeomorphic maximal MV-spectra. This result is an MV-algebraic generalization of a classical theorem by Kaplansky stating that two compact Hausdorff spaces are homeomorphic if, and only if, the lattices of continuous $[0,1]$-valued functions on the spaces are isomorphic.
\end{abstract}
\section{Introduction}
\emph{MV-algebras} were introduced by C. C. Chang \cite{Chang59} to provide algebraic semantics for {\L}ukasiewicz infinite-valued propositional logic \cite{cdm}, thus playing an analogous role to that of Boolean algebras in classical propositional logic. As proved in \cite[Thm. 3.9]{Mun1986}, MV-algebras are categorically equivalent to unital lattice-ordered abelian groups: 
the unit interval of any such group forms an MV-algebra, from which the original group and order structure can be naturally recovered; and all MV-algebras arise in this manner. MV-algebras also entertain a wide range of connections with other realms of mathematics, including probability theory, C$^*$-algebras and polyhedral geometry; see \cite{Mun2011} for a recent account. 

In this paper we study the \emph{Stone-Priestley dual spaces} of MV-algebras, by which we mean the Stone-Priestley dual spaces of the underlying lattices of MV-algebras equipped with the structure coming from the MV-algebraic operations.
Stone duality \cite{Sto1936} is of central importance in the theory of Boolean algebras: the category of Boolean algebras and their homomorphisms is equivalent to the opposite of the category of  compact Hausdorff zero-dimensional spaces---known as \emph{Boolean spaces}---and the continuous maps between them. Stone himself extended this result to distributive lattices in \cite{Stone1937}, obtaining a duality with the category of spectral spaces and so-called spectral maps.  Later, Priestley \cite{Priestley1970} obtained a useful variant of Stone duality for distributive lattices which uses \emph{Priestley spaces}, i.e.\ Boolean spaces equipped with a partial order having the topological separation property of being `totally order disconnected'. 

There are several  dualities for MV-algebras and lattice-ordered groups (henceforth abbreviated as ``$\ell$-groups'') that generalize Stone duality for Boolean algebras. 
These results can be roughly divided into two strands. In the first strand one views MV-algebras as groups with a compatible lattice order, in light of the categorical equivalence mentioned above, and looks for representations by continuous real-valued functions in the style of the Stone-Gelfand duality theory for real or commutative C$^{*}$-algebras.
At the center of this first strand, there are the Baker-Beynon duality between finitely presented abelian $\ell$-groups and homogeneous rational polyhedral sets \cite{Baker1968, Beynon1977}, and the affine version of this result which shows that finitely presented MV-algebras are the coordinate algebras of rational polyhedra under piecewise linear maps that preserve the arithmetic structure of the polyhedra in an appropriate sense  \cite{MaSpa2012, MaSpa2013}. 
These theorems highlight the profound relationship between abelian $\ell$-groups, MV-algebras, and piecewise linear arithmetic topology; for further recent manifestations of this phenomenon, see e.g.\ \cite{Ma2013, MaMaMu2007, CabMun2012, Ma2009, MaMu2007, Pan2008}. Dualities for other classes of MV-algebras and $\ell$-groups that fall into this first strand can be found in \cite{CiMa2012, CiDuMu2004, CiMu2006}. 

The second strand of duality results for MV-algebras arises from viewing MV-algebras as distributive lattices with two additional operations $\neg$ and $\oplus$ and investigating the additional structure that these operations yield on the Stone-Priestley dual space of the distributive lattice underlying the MV-algebra. This approach can be traced back several decades in the literature on duality for MV-algebras and $\ell$-groups, see e.g. \cite{Bel1986, Ma1990, Ma1994, Ma1996, MaPr1998, GePr07a,GePr07b}. In particular, \cite{Ma1996} establishes a duality theorem between $\ell$-groups and a class of spaces dubbed ``$\ell$-spaces''. The recent papers \cite{GePr07a, GePr07b} show that first-order axioms on the Stone-Priestley duals will capture a large class of varieties including MV-algebras, but these papers do not consider MV-algebras specifically. In contrast with the first strand, it seems that until now there have not been significant applications of dualities of the second strand in the subsequent literature on MV-algebras and $\ell$-groups.

In this paper, we will revisit this second strand of  dualities for MV-algebras. In particular, we apply the general insights gained in \cite{GePr07a, GePr07b} to the theory of MV-algebras {\it per se}, providing a systematic approach to sheaf representations of MV-algebras by means of their dual spaces. On the way to these results we  obtain an MV-algebraic generalization of a classical theorem by Kaplansky on the lattice of continuous real-valued functions on a compact Hausdorff space \cite{kaplansky}. We will now explain our methodology and the ensuing applications in further detail.

A large part of the mathematical machinery used in this paper originates in the theory of \emph{canonical extensions}, an essential tool in the study of additional operations relative to Stone-Priestley duality. Canonical extensions were introduced by J\'onsson and Tarski for Boolean algebras \cite{JT1952} and generalized to distributive lattices by the first-named author and J\'onsson in \cite{GehJon1994}. Canonical extensions have been used in the study of modal logic  to provide algebraic proofs of the canonicity of modal axioms; see e.g.\ \cite{GNV2005}. The first-named author and Priestley showed in \cite{GePr00} that one of the defining axioms of MV-algebras (specifically, equation (\ref{mvlaw}) in Section~\ref{ssec:mv} below) is not canonical. The papers \cite{GePr07a,GePr07b} studied canonical extensions and duality for double quasi-operators, i.e.\ operations that preserve or reverse $\vee$ and $\wedge$ in each coordinate, of which the MV-algebraic operations $\neg$ and $\oplus$ are prime examples. We rely on the results obtained in \cite{GePr07a,GePr07b}, with particular focus on their repercussions for MV-algebras. In particular, these results enable us to prove in Proposition~\ref{prop:plusprops} that the dual Stone-Priestley space of any MV-algebra has the structure of a topological partial commutative ordered semigroup.

In order to study \emph{sheaf representations}, we develop a new result in duality theory, which is of interest in its own right and may very well be applicable to other varieties of distributive lattice expansions beyond MV-algebras. In particular, in Theorem~\ref{thm:decompositiongivessheaf} we relate certain sheaf representations of distributive lattices to \emph{\'etale decompositions} of their Stone-Priestley dual spaces, generalizing results from the first-named author's PhD thesis \cite{Ge1991}. In addition, we show that this theorem lifts to distributive lattices with additional operations. These results are best phrased in the setting of \emph{stably compact spaces}, a class of topological spaces which are, roughly speaking, `stable' under the operation of taking the \emph{co-compact dual}; see  \cite{Jung2004, Lawson2010} and Theorem~\ref{thm:stabcompsp} below for details. For now, suffice it to  say that stably compact spaces form a proper generalization of spectral spaces, as they are no longer required to have a basis consisting of compact-open sets. In particular, all compact Hausdorff spaces are stably compact. The setting of stably compact spaces therefore allows us to uniformly treat sheaf representations over both spectral and compact Hausdorff spaces. Using canonical extensions, we show in the remainder of the paper that dual spaces of MV-algebras always admit \'etale decompositions in the sense of Theorem~\ref{thm:decompositiongivessheaf}, thus providing a new view of sheaf representations for MV-algebras. 

Sheaf representations for $\ell$-groups and MV-algebras originate with Keimel \cite{Keimel71}, \cite[Chapt. 10]{BKW1977}, who established {\it inter alia} a result for unital abelian $\ell$-groups that translates in a standard manner to the following result for MV-algebras: every MV-algebra is isomorphic to the global sections of a sheaf of local MV-algebras on its spectral space of prime MV-ideals equipped with the spectral topology. Here, an MV-algebra is \emph{local} if it has exactly one maximal ideal. From Keimel's work  one easily obtains  a compact Hausdorff representation, by restricting the base space to the subspace of maximal MV-ideals. In this form, the result was first proved for MV-algebras by Filipoiu and Georgescu \cite{FilGeo1995}; their paper is independent of Keimel's  treatment of $\ell$-groups. One can also consider the collection of prime MV-ideals equipped with the co-compact dual of the spectral topology.  This base space then affords a sheaf representation whose stalks are totally ordered MV-algebras, which may be regarded as easier to understand than local MV-algebras. The price one pays for this simplification is that the base space is a spectral space which is not $T_1$ in general, so that it has a non-trivial specialization order. Such a  representation for abelian (not necessarily unital) $\ell$-groups was established  by Yang in his PhD thesis \cite[Prop. 5.1.2]{Yang06}; and in \cite[Rem. 5.3.11]{Yang06} the author remarked that the results do translate to MV-algebras via the  categorical equivalence mentioned above. In Poveda's PhD thesis \cite[Thm. 6.7]{Poveda07}, the same sheaf representation of MV-algebras via the co-compact dual of the spectral topology on prime MV-ideals was obtained. The interested reader is also referred to the papers \cite{RuYa2008, DuPo2010}.  For a further recent proof of the sheaf representation for $\ell$-groups via the co-compact dual of the spectral topology, see  \cite{Schwartz2013}. For a more thorough historical account of sheaf representations in universal algebra, with particular attention to lattice-ordered groups and rings, see \cite{Kei2013} and the references therein.

Apart from giving a unified account of these results on sheaf representations for MV-algebras, we also use our analysis of the dual space of an MV-algebra to show that MV-algebras with isomorphic underlying lattices have homeomorphic maximal MV-spectra equipped with the spectral topology (Theorem~\ref{thm:kaplansky}). This result generalizes a theorem by Kaplansky \cite{kaplansky} on lattices of continuous functions; see Section~\ref{sec:kaplansky} below for a more detailed comparison.

\vspace{3mm}

{\bf Outline of the paper.} 
In Section~\ref{s:stone-priestley} we fix notation and recall topological prerequisites on Stone-Priestley duality and stably compact spaces. 
In Section~\ref{sec:sheaf} we prove the duality-theoretic results on \'etale decompositions of a dual space and sheaf representations of the corresponding algebra.
In Section~\ref{ssec:mv} we discuss background on MV-algebras, and in Section~\ref{sec:canext} we recall the necessary preliminaries on canonical extensions. 
The  analysis of the dual space of an MV-algebra, carried out in Section~\ref{s:mvdual}, allows us to obtain in Section~\ref{sec:decomp} an \'etale decomposition of the dual Stone-Priestley space of an MV-algebra indexed by the prime MV-spectrum, in the sense of Section~\ref{sec:sheaf}. 
In Section~\ref{sec:kaplansky} we prove the generalization of Kaplansky's theorem (Theorem~\ref{thm:kaplansky}). 
By combining the results of Sections~\ref{sec:sheaf} and \ref{sec:decomp}, in Section~\ref{sec:spectralproduct} we establish both  sheaf representations of an MV-algebra discussed above. As related topics we also discuss the  Chinese Remainder Theorem for MV-algebras (Subsection~\ref{ss:CRT}) and the construction of an explicit MV-algebraic term representing the sections occurring in the sheaf representations (Subsection~\ref{subsec:termdef}). 
Section~\ref{sec:further} points out  directions for further research.
\section{Stone-Priestley duality and stably compact spaces}\label{s:stone-priestley}
 This first part of this section contains the needed background on Stone-Priestley duality for distributive lattices. The second part of the section introduces stably compact spaces, the natural setting for the sheaf representations obtained in the sequel. For more on these two topics see  \cite{Priestley1994} and \cite{Jung2004, Lawson2010}, respectively, and the references therein.

All distributive lattices in this paper are assumed to have a minimum and a maximum element, denoted by $0$ and $1$, respectively, and lattice homomorphisms are assumed to preserve them. Throughout, `ideal' means `lattice ideal', and `filter' means `lattice filter'. We use $D$ to denote a distributive lattice, and we write $X(D)$ for the collection of homomorphisms $D\to\{0,1\}$. Then $X(D)$ is the  underlying set of the dual space of $D$. 
 Some authors identify $X(D)$ with the set of prime filters of $D$, while others with the set of prime  ideals. The difference is immaterial, since the map $F \longmapsto F^c := \{d \in D \mid d \not \in F\}$ is a bijection between the prime  filters and the prime  ideals of $D$.  We adopt the convention of calling the elements of $X(D)$ the \emph{points} of the dual space of $D$; we denote points in $X(D)$ by  $x, y, z, \ldots$ We denote the prime  ideal (respectively, prime filter) that  $x \in X(D)$ represents by $I_x$ (respectively, $F_{x}$). Thus, $I_{x}$ ($F_{x}$) is the inverse image of $0$ (respectively, $1$) under the homomorphsim $x$. When the lattice $D$ is clear from the context we write $X$ instead of $X(D)$.
The \emph{partial order} $\leq$ on $X$ is defined by inclusion of prime ideals, or reverse inclusion of prime filters, i.e.,
\begin{align}\label{eq:orderonX}
 x \leq y \iff I_x \subseteq I_y \iff F_x \supseteq F_y.
 \end{align}

For an element $a \in D$, we define a subset $\hat{a}$ of $X$ by
\begin{align}\label{eq:stonemap} \hat{a} := \{x \in X \ | \ a \in F_x\} = \{x \in X \ | \ a \not\in I_x\}.\end{align}
The assignment $a \longmapsto \hat{a} \subseteq X$, for $a \in D$, is often referred to as the \emph{Stone map}. Note that $\hat{a}$ is a \emph{downset} in the order $\leq$, that is, if $x \in \hat{a}$ and $x' \leq x$, then $x' \in \hat{a}$. Its set-theoretic complement $\hat{a}^{\, c}$ is therefore an \emph{upset}, in the obvious sense.

There are three relevant topologies on the set  $X$.
\begin{enumerate}
\item The \emph{spectral} or  \emph{Stone}  or \emph{Zariski topology} $\tau^{\downarrow}$, given by taking $\{\hat{a} \ | \ a \in D\}$ as a basis for the open sets. 
\item The \emph{co-spectral} or  \emph{dually spectral} or \emph{co-Zariski topology} $\tau^{\uparrow}$, given by taking $\{\hat{a}^{\, c} \ | \ a \in D\}$ as a basis for the open sets.\footnote{
This topology coincides with the \emph{co-compact dual} or \emph{De Groot-Hochster dual} of the spectral topology, cf. Theorem~\ref{thm:stabcompsp} below, also see\ \cite{DeGroot1967, Hochster1969}.
}
\item The \emph{Priestley} or \emph{patch topology} $\tau^p$, given by taking as a subbasis for the open sets the collection $\{\hat{a} \ | \ a \in D\} \cup \{(\hat{a})^c \ | \ a \in D\}$.
\end{enumerate}
Each one of
$(X,\tau^{\downarrow})$, $(X,\tau^{\uparrow})$, and $(X,\tau^{p}, \leq)$
is liable to be called the \emph{dual space of $D$} in the literature. We always make it clear  which ``dual space'' we are referring to. Also, throughout, for  $S\subseteq X$ we write $(S,\tau^p,\leq)$, $(S,\tau^\downarrow)$ and $(S,\tau^\uparrow)$ to denote the structures on $S$ with the relative topology and order inherited from $X$. 

Recall that a topological space $(X,\tau)$ is  \emph{spectral}\footnote{
These spaces are called \emph{coherent spaces} in \cite{Johnstone1982}, but we reserve this name for a strictly weaker property, following  e.g.\ \cite{Jung2004,Lawson2010}.}
if it is compact, sober,\footnote{
A topological space is  \emph{sober} \cite[Chapt. II, Sec. 1.6]{Johnstone1982}  if any closed irreducible subset is the closure of a unique point. Here, a closed subset is \emph{irreducible} if it can not be written as a union of two closed proper subsets. Note that any sober space is $T_0$, and so are all spaces we consider in this paper.}
and its compact open sets form a basis which is closed under taking finite intersections. A map $f : (X, \tau) \to (X',\tau')$ between spectral spaces is called {\it spectral} if the inverse image under $f$ of any compact open set is compact. Note in particular that spectral maps are continuous. A topological space $(X,\tau)$ is  \emph{Boolean}\footnote{
A \emph{Stone space} in \cite{Johnstone1982}.}
if it is spectral and $T_2$, or equivalently (see, e.g., \cite[Chapt.II, Sec.4.2]{Johnstone1982}), if it is compact, $T_2$, and its clopen sets form a basis. An ordered topological space $(X,\tau,\leq)$ is  \emph{totally order-disconnected} provided that, whenever $x \nleq y$, there is a clopen downset $U \subseteq X$ such that $y \in U$ and $x \not\in U$. A compact totally order-disconnected space $(X,\tau,\leq)$ is  a \emph{Priestley space}; its underlying space $(X,\tau)$ is always Boolean \cite[Sec. 3]{Priestley1970}.
\begin{dualitythm}[\cite{Stone1937,Priestley1970}]
The category of distributive lattices with homomorphisms is dually equivalent to each of the following categories.
\begin{enumerate}
\item The category of spectral spaces with spectral maps.
\item The category of Priestley spaces with continuous order-preserving maps.
\end{enumerate}
\end{dualitythm}
\noindent We recall the construction of the pairs of contravariant functors $(F^s,G^s)$ and  $(F^p,G^p)$ leading to the dual equivalences 1 and 2 above, respectively.
\begin{itemize}

\item The functor $F^s$ sends a distributive lattice $D$ to the space $(X(D),\tau^\downarrow)$, and a lattice homomorphism $h : D \to E$ to the function $f \colon X(E) \to X(D)$ which sends $x \in X(E)$ to the point $f(x) \in X(D)$ for which $I_{f(x)} = h^{-1}(I_x)$. 

\item The functor $G^s$ sends a spectral space $(X,\tau)$ to the distributive lattice of compact open subsets of $X$, ordered by inclusion, and a spectral map $f \colon (X,\tau) \to (X',\tau')$ to the homomorphism $f^{-1}$ from $G^s(X',\tau')$ to $G^s(X,\tau)$.
\end{itemize}
\begin{itemize}
\item The functor $F^p$ sends a distributive lattice $D$ to the ordered topological space $(X(E),\tau^p,\leq)$, and a homomorphism $h \colon D \to E$ to the same function $f \colon X(E) \to X(D)$ as in the definition of $F^s$, now regarded as a function between ordered topological spaces. 
\item The functor $G^p$ sends a Priestley space $(X,\tau,\leq)$ to the distributive lattice of clopen downsets of $X$, ordered by inclusion, and a continuous order-preserving map $f \colon (X,\tau,\leq) \to (X',\tau',\leq')$ to the homomorphism $f^{-1}$ from $G^p(X',\tau',\leq')$ to $G^p(X,\tau,\leq)$.
\end{itemize}
\begin{cor}\label{cor:prdualobjects}
Any distributive lattice $D$ is naturally isomorphic to the lattice of compact open subsets of $(X,\tau^\downarrow)$, and to the lattice of clopen downsets of $(X,\tau^p,\leq)$.
\end{cor}
\begin{proof}By the Stone-Priestley Duality Theorem, the  isomorphisms are given by the natural maps $\eta^s_D \colon D \to G^sF^s(D)$ and $\eta^p_D \colon D \to G^pF^p(D)$, respectively. Concretely, these maps are both equal to the Stone map defined above.
\end{proof}
\begin{cor}\label{cor:isohomeo}
The isomorphisms between two distributive lattices $D$ and $E$ are in natural 1-1 correspondence with the homeomorphisms between the dual spaces $(X(D),\tau^\downarrow)$ and $(X(E),\tau^\downarrow)$, and also with the order-preserving and order-reflecting homeomorphisms between the dual spaces $(X(D),\tau^p,\leq)$ and $(X(E),\tau^p,\leq)$.
\end{cor}
\begin{proof}
Immediate consequence of the Stone-Priestley Duality Theorem, upon noting that the isomorphisms in the category of Priestley spaces are exactly the order-preserving and order-reflecting homeomorphisms.
\end{proof}
If $(X,\tau)$ is a topological space, let us write $\leq_{\tau}$ for its \emph{specialization order}: by definition,  $x \leq_\tau y$ if, and only if, $x \in \overline {\{y\}}$, where $\overline{\; \cdot \;}$ denotes the closure operator induced by $\tau$; equivalently, for any open set $U$, if $x \in U$, then $y \in U$. The relation $\leq_{\tau}$ is a partial order on $X$ precisely when $\tau$ is $T_{0}$, and this is always the case for $\tau^{\downarrow}$, $\tau^{\uparrow}$, and $\tau^{p}$---the first two are sober, the last  is  $T_2$. 
 Recall that a partial order $\leq$ on a dual space $X(D)$ of a distributive lattice $D$ was defined in equation (\ref{eq:orderonX}) on p. \pageref{eq:orderonX}. The definitions yield the following at once.
\begin{itemize}
\item The specialization order of $\tau^\uparrow$ is equal to the partial order $\leq$ on $X$, that is, $\leq \; = \; \leq_{\tau^\uparrow}$. 
\item The specialization order of $\tau^\downarrow$ is the opposite of  the specialization order of $\tau^\uparrow$. 
\item The specialization order of $\tau^p$ is trivial. Indeed, $\tau^p$ is a $T_2$ topology, and hence {\it a fortiori} $T_1$, which means that the specialization order is trivial.\end{itemize}
The  categorical equivalence between spectral  and Priestley spaces given by the Stone-Priestley Duality Theorem yields the following.
\begin{itemize}
 \item The topology $\tau^\downarrow$ is precisely the collection of open downsets of the ordered space $(X,\tau^p,\leq)$.
\item The topology  $\tau^\uparrow$ is precisely the collection of open upsets of the ordered space $(X,\tau^p,\leq)$.
\item The  relation $\leq$ is a closed subset of the product space $(X,\tau^\uparrow) \times (X,\tau^\downarrow)$.
\end{itemize}
\noindent The Stone map (\ref{eq:stonemap}) can be extended from elements to filters and ideals of $D$:
\begin{align}\label{eq:stonemap2} 
F \subseteq D \text{ a filter } &\longmapsto \bigcap_{a \in F} \hat{a} = \{x \in X \ | \ F \subseteq F_x\} = \{x \in X \ | \ F \cap I_x = \emptyset\}, \\
\label{eq:stonemap3} 
I \subseteq D \text{ an ideal } &\longmapsto \bigcup_{a \in I} \hat{a} = \{x \in X \ | \ I \cap F_x \neq \emptyset\} = \{x \in X \ | \ I \not\subseteq I_x\}.
\end{align}
Under (\ref{eq:stonemap2}--\ref{eq:stonemap3}), filters and ideals of $D$     correspond  to  subsets of $X$ which can be characterized topologically, see  Table \ref{t:corr} below. This correspondence extends the one in Corollary \ref{cor:prdualobjects} upon identifying an element $a\in D$ with the principal ideal (or filter) of $D$ that it generates. We recall that a subset of a topological space is called {\em saturated} if it is an intersection of open sets, and {\em co-saturated} if it is a union of closed sets; also, we say \emph{co-compact set} to mean \emph{complement of a compact set}.\footnote{The term `co-compact' is often (e.g.\ in \cite{Jung2004, Lawson2010}) used to refer to `complement of a compact saturated set'; we will call such sets `co-compact co-saturated', to avoid possible confusion.}

\begin{table}[htp]
\begin{small}
\begin{center}
\begin{tabular}{|c|c|c|c|}
\hline
 $D$ & $(X,\tau^\downarrow)$ & $ (X,\tau^p,\leq)$ & $(X,\tau^\uparrow)$ \\ \hline\hline
element 			  &   compact open 		 & clopen downset  & co-compact closed \\ \hline
lattice filter $F$				  &  compact saturated       & closed downset  & closed \\ \hline
lattice ideal  $I$                          &   open                             & open downset   & co-compact co-saturated\\
\hline
\end{tabular}
\end{center}
\caption{Correspondence between lattice and space concepts}
\label{t:corr}
\end{small}
\end{table}

\vspace{-5mm}

Quotients of $D$ dually correspond to closed subspaces of $(X,\tau^p)$ by the Stone-Priestley Duality Theorem. We record details for later use.
\begin{prop}
\label{prop:quotientsubspace}
Let $D$ be a distributive lattice and let $(X,\tau^p,\leq)$ be its Priestley dual space. There is a contravariant Galois connection between the subsets of $D \times D$ and the subsets of $X$, given by
\begin{align*}
\left. \theta \subseteq D \times D \right. \ \ &\longmapsto \ \ \left. S_\theta := \left\{x \in X  \mid \forall (a,b) \in \theta \,\, (a \in I_x \liff b \in I_x)\right\}, \right. \\
\left. S \subseteq X \right. \ \ &\longmapsto \ \  \left. \theta_S := \left\{(a,b) \in D \times D  \mid  \forall x \in S \,\, (a \in I_x \liff b \in I_x)\right\}.\right.
\end{align*}
Moreover, for $\theta \subseteq D \times D$, we have $\theta = \theta_{S_\theta}$ if, and only if, $\theta$ is a distributive lattice congruence. Also, for $S \subseteq X$, we have $S = S_{\theta_S}$ if, and only if, $S$ is closed in $\tau^p$.
In particular, if $D \onto D/\theta$ is a  quotient of $D$, then the Priestley dual space of $D/\theta$ is isomorphic to $(S_\theta,\tau^{p},\leq)$. \end{prop}

We now turn to stably compact spaces. A topological space is {\it locally compact} if, given $U$ open and $x \in U$, there exist $V$ open and $K$ compact such that $x \in V \subseteq K \subseteq U$. A space is {\it coherent} if the intersection of two compact saturated sets is again compact, and {\it well-filtered} if, whenever the intersection of a descending chain of compact saturated sets is contained in an open set, $U$, then some member of the chain is already contained in $U$. A space is {\it stably compact} if it  is $T_0$, compact, locally compact, coherent and well-filtered. In this definition, the property `well-filtered' may be equivalently replaced by `sober', cf.\ \cite[Rem. 2.4]{Lawson2010}.

\begin{exa}\label{exa:stabcompsp}
 Any compact Hausdorff space is stably compact.
({\it Proof}.
Compact Hausdorff spaces are locally compact \cite[Thm. 18.2]{Willard}, and their compact  and closed sets coincide \cite[Thm. 17.5]{Willard}. From this it follows easily  that compact Hausdorff spaces are coherent and well-filtered.)
Also, any spectral space is stably compact.
({\it Proof}.
Spectral spaces are compact, sober, and $T_0$ by definition. Since the compact open sets form a basis for the topology, spectral spaces are locally compact. Further, the compact saturated sets in a spectral space are exactly the intersections of compact open sets. From this, it follows that a spectral space is coherent.) Indeed, spectral spaces are precisely the stably compact spaces in which the compact open sets form a basis.
\end{exa}

If $(I,\rho)$ is a topological space, then the {\it co-compact dual topology} $\rho^\dual$ on $I$ is defined by taking as a basis for the open sets the complements of $\rho$-compact-saturated sets. For a stably compact topology $\rho$, the $\rho$-compact-saturated sets are closed under arbitrary intersections (see e.g.\  \cite[Lem. 2.8]{Jung2004}), from which it follows that the open sets in $\rho^\dual$ are {\it exactly} the complements of the $\rho$-compact-saturated sets. The {\it patch topology} $\rho^p$ is defined as the smallest topology containing both $\rho$ and $\rho^\dual$, that is, $\rho^p := \rho \vee \rho^\dual$.  Observe that, for a Priestley space $(X,\tau^p,\leq)$, $\tau^{\uparrow}=(\tau^{\downarrow})^{\dual}$, and $\tau^{p}=\tau^{\downarrow}\vee\tau^{\uparrow}$.
\begin{thm}\label{thm:stabcompsp}
Let $(I,\rho)$ be a stably compact space with specialization order $\leq$. The following hold.
\begin{enumerate}
\item The space $(I,\rho^\dual)$ is stably compact.
\item The order $\leq$ is a closed subset of the product space $(I,\rho^\dual) \times (I,\rho)$.
\item The open sets of $(I,\rho)$ are exactly the open downsets of the ordered space $(I,\rho^p,\leq)$.
\item The open sets of $(I,\rho^\dual)$ are exactly the open upsets of the ordered space $(I,\rho^p,\leq)$.
\item If the order $\leq$ is discrete, then $\rho^\dual = \rho^p = \rho$.
\end{enumerate}
\end{thm}
\begin{proof}
For item 1, see \cite[Cor. 2.13(i)]{Jung2004}. For item 2, see \cite[Lem. 2.2(ii)]{Jung2004}. For items 3 and 4, see \cite[Thm. 2.12]{Jung2004}. Item 5 follows immediately from items 3 and 4 and the definition of $\rho^p$.
\end{proof}
Note in particular that, for a stably compact space $(I,\rho)$, $(\rho^\dual)^\dual = \rho$; hence the term `stable'.
\begin{prop}\label{prop:downofcompact}
Let $(I,\rho)$ be a stably compact space with specialization order $\leq$. If $K \subseteq I$ is compact in the topology $\rho$, then ${\downarrow} K$ is closed in $(I,\rho^\dual)$.
\end{prop}
\begin{proof}
By Theorem~\ref{thm:stabcompsp}.2, $\leq$ is a closed subset of the product space $(I,\rho^\dual) \times (I,\rho)$, so the intersection of $\leq$ with the subspace $C:=I \times K$ is also closed. Because $(K,\rho)$ is compact, the projection map  $\pi \colon (I,\rho^\dual) \times (K,\rho) \to (I,\rho^\dual)$ is closed, by a standard topological fact (see e.g.\ \cite[Lem. 26.8, Exercise 26.7]{Munk2000}). Hence ${\downarrow} K = \pi[C]$ is closed in $(I,\rho^\dual)$.
\end{proof}
\begin{cor}\label{cor:upofcompact}
Let $(I,\rho)$ be a stably compact space with specialization order $\leq$. If $K \subseteq I$ is compact in the topology $\rho^\dual$, then ${\uparrow} K$ is closed in $(I,\rho)$.
\end{cor}
\begin{proof}
By Theorem~\ref{thm:stabcompsp}.1, $(I,\rho^\dual)$ is  a stably compact space, too, with  specialization order $\geq$. Now apply Proposition~\ref{prop:downofcompact}.
\end{proof}
\section{Sheaf representations via duality}\label{sec:sheaf}
In this section we  show how to obtain a sheaf representation of an algebra with a distributive-lattice reduct via a decomposition of its dual space. The main result of this section, Theorem~\ref{thm:decompositiongivessheaf}, generalizes \cite[Thm. 3.7 and Thm. 3.12]{Ge1991}. In Section~\ref{sec:spectralproduct}, we will apply its Corollary \ref{cor:dlesheaf}  to the variety of MV-algebras.

We  recall some basic notions from sheaf theory; see e.g. \cite{Tennison75, MacLaneMoerdijk94} for background. Let $\V$ be a variety of algebras, and let $(I,\rho)$ be a topological space. Regarding $\rho$ as a category with arrows given by inclusions, consider a (presheaf) functor $F\colon \rho^{\rm op} \to \V$. As usual, we denote the image under $F$ of an inclusion $U \subseteq V$ in $\rho$ by $(-)|_U \colon F(V) \to F(U)$, and call it the {\it restriction map} from $V$ to $U$. Suppose further that $F$ is a sheaf such that the $\V$-algebra $A$ is isomorphic to the algebra of global sections of $F$. Recall that, for each $i \in I$, the {\it stalk} $A_i$ of $F$ at $i$ is  the direct limit (=filtered colimit) of the algebras $F(U)$, where $i \in U$. Concretely, $A_i$ is the quotient of $A$ by the congruence relation  given by
\begin{align*}
a \sim_i b \iff \exists U \subseteq I \text{ open such that } i \in U \text{ and } a|_U = b|_U.
\end{align*}
In particular, consider the case where $\V$ is the variety of distributive lattices, and $F$ is a sheaf whose algebra of global sections is $D$. Each of the stalk quotients $D\onto D_{i}$ corresponds to a closed subspace $X_i \embeds X$ of the Priestley dual space of $D$ (Proposition~\ref{prop:quotientsubspace}), namely,
\begin{align*} 
X_i := \{x \in X \ | \ \forall a,b \in D , \text{ if } a \sim_i b, \text{ then } x \in \hat{a} \liff x \in \hat{b}\}.
\end{align*}
Thus, the sheaf $F$ yields a collection $(X_i)_{i\in I}$ of closed subspaces  of $X$ indexed by $(I,\rho)$. 
Conversely, given a Priestley space $(X,\tau^p,\leq)$, we  seek conditions guaranteeing  that a collection  $(X_i)_{i \in I}$ of subspaces of $X$  indexed by $(I,\rho)$  yields a sheaf-theoretic representation of the distributive lattice $D$ dual to $(X,\tau^p,\leq)$. Theorem~\ref{thm:decompositiongivessheaf} below will provide such conditions in case $(I,\rho)$ is stably compact, and Corollary~\ref{cor:dlesheaf} will show that any additional operations on $D$ are  automatically preserved by the construction.

Let $(X,\tau^p,\leq)$ be a Priestley space and $D$ its dual distributive lattice. Suppose  $q \colon (X,\tau^p,\leq) \to (I,\rho)$ is a continuous surjection onto the stably compact space $(I,\rho)$, whose specialisation order we denote $\leq_\rho$.  With the aim of associating an \'etale space to $q$, we observe the following. By Corollary~\ref{cor:upofcompact},  ${\uparrow}i$ is closed for any $i \in I$,  whence $X_i := q^{-1}({\uparrow}i)$ is closed. Let $f_i \colon D \onto D_i$ be the quotient corresponding to $X_{i}$ via Proposition~\ref{prop:quotientsubspace}, so that the kernel of $f_i$ is the congruence $\theta_i$ defined by
\begin{align*} a \, \theta_i \, b \iff \hat{a} \cap X_i = \hat{b} \cap X_i.
\end{align*}
Now, for any $i, j \in I$, if $i \leq_\rho j$ then $X_j \subseteq X_i$ and $\theta_i \subseteq \theta_j$. (Indeed, $i \leq_\rho j$ entails ${\uparrow} j \subseteq {\uparrow} i$, so $X_j \subseteq X_i$. By Proposition~\ref{prop:quotientsubspace}, this is equivalent to $\theta_i \subseteq \theta_j$.)
Thus, whenever $i \leq j$, we have a quotient map $f_{i,j} \colon  D_i \onto D_j$ which corresponds under Priestley duality to the subspace inclusion $X_j \embeds X_i$, namely, the unique map such that $f_{i,j} \circ f_i = f_j$.

We now standardly  construct an \'etale space whose stalks are the distributive lattices $D_i$. Let $E$  denote the disjoint union of the sets $D_i$, for $i \in I$, and let $p \colon E \to I$ be the map $[a]_{\theta_i} \in D_i\mapsto i\in I$. Any $a \in D$ has an associated global section $s_a \colon I \to E$ of $p$ that acts by $i \in I \mapsto  [a]_{\theta_i}\in D_{i}$. Let $\sigma$ be the topology on $E$ generated by the subbasis $\{s_a(U) \ | \ a \in D, \ U \text{ open in } \rho^\dual\}$.
\begin{prop}\label{prop:pisetalemap}
For any $a \in D$, the section $s_a \colon (I,\rho^\dual) \to (E,\sigma)$ is continuous. In particular, $p \colon (E,\sigma)\to (I,\rho^\dual)$ is a local homeomorphism.
\end{prop}
\newcommand{\sd}{\bigtriangleup}
\begin{proof}
For the first statement, it suffices to show that the inverse image under $s_a$ of a subbasic element is open in $\rho^\dual$. To this end, let $b \in D$, $U \in \rho^\dual$, and consider the inverse image $s_a^{-1}(s_b(U))$. Writing $\sd$ to denote set-theoretic symmetric difference,  for any $i \in I$ we have:
\begin{align*}
i \not\in s_a^{-1}(s_b(U)) &\iff \text{if } i \in U, \text{ then } [a]_{\theta_i} \neq [b]_{\theta_i}, \\
&\iff \text{if } i \in U, \text{ then } \hat{a} \cap X_i \neq \hat{b} \cap X_i, \\
&\iff \text{if } i \in U, \text{ then } \exists x \in X_i \text{ such that } x \in \hat{a} \sd \hat{b}, \\
&\iff i \in U^c \cup {\downarrow}q[\hat{a} \sd \hat{b}].
\end{align*}
(For the last equivalence, recall that $x \in X_i$ iff $i \leq q(x)$ by definition.) Thus, $s_a^{-1}(s_b(U)) = U \cap ({\downarrow}q[\hat{a}\sd\hat{b}])^c$. We now show that $({\downarrow}q[\hat{a}\sd\hat{b}])^c$ is open in $(I,\rho^\dual)$.
The set $\hat{a} \sd \hat{b}$ is closed in $(X,\tau^p,\leq)$,  hence compact;  therefore, by continuity of $q$, the set $q[\hat{a} \sd \hat{b}]$ is compact in $(I,\rho)$. By Proposition~\ref{prop:downofcompact}, ${\downarrow}q[\hat{a} \sd \hat{b}]$ is closed in $(I,\rho^\dual)$, so its complement is open, as required.

To see that $p$ is a local homeomorphism, note first that $p$ is continuous: if $U \subseteq I$ is $\rho^\dual$-open, then $p^{-1}(U) = \bigcup_{a \in D} s_a(U)$, which is open in the topology $\sigma$. Moreover, for any $e \in E$ we have $e = [a]_{\theta_i}$ for some $i \in I$, $a \in D$. Then $s_a(I)$ is an open set around $e$, and $s_a$ is a continuous inverse to $p|_{s_a(I)}$.
\end{proof}
\begin{dfn}
For $q \colon (X,\tau^p,\leq) \onto (I,\rho)$ a continuous surjection from a Priestley space onto a stably compact space, the map $p \colon (E,\sigma) \onto (I,\rho^\dual)$ defined above is called the {\it \'etale space associated to $q$}.
\end{dfn}

Suppose $q$ and $p$ are as above, and let $\Gamma(E,p)$ denote the set of continuous global sections of $p \colon (E,\sigma) \onto (I,\rho^\dual)$. Viewing $\Gamma(E,p)$ as a subset of the direct  product $\prod_{i\in I}D_i$, it is a distributive lattice under the pointwise operations. We thus have a well-defined homomorphism $\eta \colon  D \to \Gamma(E,p)$ given by $\eta(a) := s_a$. Moreover, {\it the map $\eta$ is injective}: if $a, b \in D$ and $a \neq b$, then, by Priestley duality, there exists $x \in X$ such that $x$ is in exactly one of $\hat{a}$ and $\hat{b}$. Let $i := q(x)$. Then in particular $\hat{a} \cap X_i \neq \hat{b} \cap X_i$, so $f_i(a) \neq f_i(b)$. Hence, $s_a \neq s_b$, that is, $\eta(a) \neq \eta(b)$, as claimed. When the embedding $\eta \colon  D \embeds \Gamma(E,p)$ is moreover surjective, and therefore an isomorphism of distributive lattices, we say that $p \colon (E,\sigma) \onto (I,\rho^\dual)$ is  a {\it sheaf representation of the distributive lattice} $D$.
\begin{thm}\label{thm:decompositiongivessheaf}
Let $D$ be a distributive lattice with dual Priestley space $(X,\tau^p,\leq)$. Suppose that $q \colon  (X,\tau^p,\leq) \onto (I,\rho)$ is a continuous surjection onto a stably compact space which satisfies the following \textup{patching property}.
\begin{enumerate}
\item[\textup{(P)}] 
 Let $(U_l)_{l=1}^n$ any finite cover of $I$ by $\rho^\dual$-open sets, and let $(K_l)_{l=1}^n$  be any finite collection of clopen downsets of $X$ such that
\begin{align*}\tag{*}\label{label:star}
K_l \cap q^{-1}(U_l \cap U_m) = K_m \cap q^{-1}(U_l \cap U_m)
\end{align*}
holds for any $l,m \in \{1,\dots,n\}$. Then the set $\bigcup_{l=1}^n (K_l \cap q^{-1}(U_l))$ is a clopen downset in $X$.
\end{enumerate}

\noindent Then the associated \'etale space $p \colon  (E,\sigma) \onto (I,\rho^\dual)$ is a sheaf representation of $D$.
\end{thm}
\begin{proof}
As shown above, we only need to prove that $\eta$ is surjective. Let $s \in\Gamma(E,p)$ be a continuous section of $p$. For each $i \in I$, choose $a_i \in A$ such that $s(i) = f_i(a_i)$. Then each $U_i := s^{-1}(s_{a_i}(I))$ is $\rho^\dual$-open. Since $(I,\rho^\dual)$ is compact, there is a finite subcover $(U_{l})_{l=1}^{n}$ of $(U_{i})_{i\in I}$; write $a_1, \dots, a_n$ for the corresponding elements of $A$, and set $K_l := \hat{a}_l$ for $l = 1, \dots, n$. We now prove that $(U_l)_{l=1}^n$ and $(K_l)_{l=1}^n$ satisfy (\ref{label:star}). Indeed, if $x \in \hat{a_l} \cap q^{-1}(U_l \cap U_m)$, then $i := q(x) \in U_l$, so $s(i) = s_{a_l}(i)$, but also $s(i) = s_{a_m}(i)$. Therefore, $[a_l]_{\theta_i} = [a_m]_{\theta_i}$. In particular, $x \in \hat{a}_l \cap X_i = \hat{a}_m \cap X_i$, so $x \in \hat{a}_m$, and (\ref{label:star}) holds. Now, by property (P), the set $\bigcup_{l=1}^n (K_l \cap q^{-1}(U_l))$ is a clopen downset. By Priestley duality, pick $a \in A$ such that $\hat{a} = \bigcup_{l=1}^n (K_l \cap q^{-1}(U_l))$. We show that $s = s_a$. Pick $i \in I$, and $l \in \{1,\dots,n\}$ such that $i \in U_l$, so that $s(i) = [a_l]_{\theta_i}$. We now show that $\hat{a}_l \cap X_i = \hat{a} \cap X_i$, which is equivalent to $[a_l]_{\theta_i} = [a]_{\theta_i}$. Since $U_l$ is an upset, we have $X_i = q^{-1}({\uparrow}i) \subseteq q^{-1}(U_l)$, so $\hat{a}_l \cap X_i = K_l \cap X_i \subseteq K_l \cap q^{-1}(U_l) \subseteq \hat{a}$, and hence $\hat{a}_l \cap X_i \subseteq \hat{a} \cap X_i$. For the converse inclusion, if $x \in \hat{a} \cap X_i$, pick $m \in \{1,\dots,n\}$ such that $x \in K_m \cap q^{-1}(U_m)$. Since $x \in X_i$, we have $x \in q^{-1}(U_l)$. Then $x \in K_m \cap q^{-1}(U_l \cap U_m) = K_l \cap q^{-1}(U_l \cap U_m) \subseteq \hat{a}_l$. Thus we have  $\eta(a) = s_a = s$, as was to be shown. 
\end{proof}

By a {\it distributive lattice expansion} we mean an algebra $A$ in a signature containing $\vee, \wedge, 0, 1$, such that the reduct of $A$ in the signature $\vee,\wedge,0,1$ is a distributive lattice. 
\begin{cor}\label{cor:dlesheaf}
Let $\V$ be a variety of distributive lattice expansions, let $A$ be an algebra in $\V$ whose distributive lattice reduct has dual Priestley space $(X,\tau^p,\leq)$. Suppose that $q\colon  (X,\tau^p,\leq) \onto (I,\rho)$ satisfies the conditions of Theorem~\ref{thm:decompositiongivessheaf}, and, moreover, that each induced lattice quotient $A \onto A_i$ is a homomorphism of $\V$-algebras. Then $\Gamma(E,p)$ is a $\V$-subalgebra of $\prod_{i\in I} A_i$ which is isomorphic to $A$ as a $\V$-algebra.
\end{cor}
\begin{proof}
The function $\eta : A \to \Gamma(E,p)$ which sends $a$ to $s_a$ is a lattice isomorphism by Theorem~\ref{thm:decompositiongivessheaf}. Any additional operation on $A$ is preserved by $\eta$, by the assumption that each map $A \onto A_i$ is a $\V$-homomorphism. So $\eta$ is a $\V$-isomorphism. In particular, $\Gamma(E,p)$ is a $\V$-subalgebra of $\prod_{i \in I} A_i$.
\end{proof}
Under the hypotheses of the preceding corollary, we say that $p$ is a {\it sheaf representation of the algebra $A$}.
\section{MV-algebras and their spectral spaces}\label{ssec:mv}
In this section, we recall the basic definitions of MV-algebras and the different (lattice, prime, maximal) spectra that have been associated to them in the literature. We also describe the relationship between these spectra and the dual spaces of the underlying distributive lattice of the MV-algebra. We  use a minimal amount of MV-algebraic theory in this paper. Indeed, almost all results that we  need can be found in \cite[Chapt. 1]{cdm}. In particular, we neither assume Chang's completeness theorem \cite[Thm. 2.5.3]{cdm}, nor even the  easier subdirect representation theorem \cite[Thm. 1.3.3]{cdm}; the latter is a straightforward consequence of the first sheaf representation given in Section~\ref{sec:spectralproduct}. 

Background references for MV-algebras are \cite{cdm,Mun2011}. We recall that an \emph{MV-algebra} is an algebraic structure $(M,\oplus,\neg,0)$, where $0\in M$ 
is a constant, $\neg$ is a unary operation satisfying $\neg\neg x=x$, $\oplus$ is a binary operation making $(M,\oplus,0)$
a commutative monoid, the element $1$ defined as $\neg 0$ satisfies $x\oplus1=1$, and the law
\begin{align}\label{mvlaw}
 \neg(\neg x \oplus y)\oplus y = \neg(\neg y \oplus x)\oplus x 
\end{align}
holds.   MV-algebras form a variety of algebras by their very definition.

Any MV-algebra has an underlying structure of 
distributive lattice bounded below by $0$ and above by $1$, \cite[Prop. 1.1.5]{cdm}. Joins are defined as $x \vee y = \neg(\neg x \oplus  y)\oplus y$. 
Thus, the characteristic law (\ref{mvlaw}) states
that $x\vee y=y\vee x$.
Meets are defined by the De Morgan equation  $x \wedge y = \neg (\neg x \vee \neg y)$. It is common to call \emph{MV-chains} those MV-algebras whose underlying order is total.
For $m\geq 1$ an integer, and $x$ an element of an MV-algebra, we often abbreviate
\begin{align*}
mx:=\underset{m \text{ occurrences of } x}{\underbrace{x\oplus\cdots\oplus x}\,.}
\end{align*}

We set, as usual,
$x\ominus y:= \neg (\neg x \oplus y)$. 
Boolean algebras are precisely those MV-algebras that are idempotent, meaning that $x\oplus x = x$ holds; equivalently, they are the MV-algebras that satisfy the {\it tertium non datur} law $x\vee\neg x=1$. For Boolean algebras, $\oplus$ is $\vee$, and $\ominus$ is the operation of logical difference, $x\wedge \neg y$.

The real unit interval  $[0,1]\subseteq \R$ can be made into an MV-algebra with neutral element $0$
by defining $x\oplus y := \min{\{x+y,1\}}$ and $\neg x:=1-x$. The underlying lattice order of this MV-algebra coincides with 
the natural order of $[0,1]$. Thus, in this example, $\oplus$ can be thought of as `truncated addition', and $x\ominus y$ as `truncated subtraction', i.e.\ 
$x\ominus y=\max{\{x-y,0\}}$. The example is generic by \emph{Chang's Completeness Theorem}:
The variety of MV-algebras is generated\footnote{Hence the subalgebra $[0,1]\cap\Q$ also generates the variety of MV-algebras: if an evaluation into $[0,1]$ makes two terms
unequal, then by the continuity of the MV-algebraic operations on $[0,1]$ there also is an evaluation into $[0,1]\cap\Q$ that makes those two terms unequal.} by the standard MV-algebra $[0,1]$.  Chang's original proof is in \cite[Lem. 8]{Chang59}; for a textbook treatment, see \cite[Thm. 2.5.3]{cdm}. As emphasized in the Introduction, however, our results are obtained independently of  Chang's theorem. Indeed, in the remainder of this subsection we collect all of the simple facts about MV-algebras that we shall use in the sequel.
Throughout this section we let $A$ denote an MV-algebra. 
An \emph{MV-ideal} of an MV-algebra $A$ is a subset $I\subseteq A$ that is a submonoid (i.e.\ contains $0$ and is closed under $\oplus$) and a downset (i.e.\ contains
$b\in A$ whenever it contains $a \in A$ and $b\leq a$). 
The MV-ideals of $A$ are in natural bijection with the congruences on $A$ by \cite[Prop. 1.2.6]{cdm}. 
We    write  $A/I$ to denote the quotient of the MV-algebra $A$ modulo the ideal $I$.

An MV-ideal $I\subseteq A$ is \emph{prime} if it is proper (i.e.\ $I \neq A$), and for each $a,b\in A$ either $(a\ominus b) \in I$, or $(b\ominus a) \in I$.\footnote{The terminology ``prime MV-ideal'' is consistent with the terminology ``prime (lattice) ideal'' of Section~\ref{s:stone-priestley}, cf. Proposition~\ref{prop:consistency} below.} An MV-ideal $I\subseteq A$ is \emph{maximal} if it is proper, and the only MV-ideal that properly contains $I$ is $A$ itself. The quotient $A/I$ is then totally ordered by elementary universal algebra;\footnote{It is specific to MV-algebras, on the other hand, that $A/I$  admits a unique embedding  into the standard MV-algebra $[0,1]$ precisely when $I$ is maximal \cite[Thm. 6.1.1]{cdm}.} hence, by Proposition \ref{prop:primechain}, maximal MV-ideals are prime. 
 Each MV-ideal of $A$ is a lattice ideal of (the underlying lattice of) $A$; cf.\ \cite[Prop. 4.13]{Mun2011}. Although the converse fails, we have:
\begin{prop}[{\cite[Lem. 6.1.1]{cdm}}]\label{prop:consistency}Let $I$ be an MV-ideal of an MV-algebra $A$. Then $I$ is a prime ideal of the underlying lattice of $A$ if, and only if, $I$ is a prime MV-ideal.
\end{prop}
\begin{prop}[{\cite[Lem. 1.2.3(v)]{cdm}}]\label{prop:primechain} If $p\colon A \to B$ is an onto homomorphism of MV-algebras, then $B$ is totally ordered and non-trivial \textup{(}i.e.\ not a singleton\textup{)} if, and only if, the  MV-ideal $p^{-1}(0)$ is prime. 
\end{prop}
We write $\langle S \rangle$ to denote the MV-ideal generated by the subset $S\subseteq A$, namely, the intersection of all MV-ideals of $A$ containing $S$.  When $S=\{s\}$ is a singleton we write $\langle s \rangle$ instead of $\langle \{s\}\rangle$, and speak of  \emph{principal MV-ideals}.
\begin{prop}\label{prop:idealgenerated} For any non-empty subset $S$ of an MV-algebra $A$, we have
\begin{align*}
\langle S \rangle = \{a \in A \mid a \leq s_1 \oplus \cdots \oplus s _{k}, \text{ for some finite set  } \{s_{i}\}_{i=1}^{k} \subseteq S \}\,.
\end{align*}
Further, for any $a,b,c\in A$ we have  $\langle a\rangle \cap \langle b\rangle = \langle a \wedge b\rangle$, and  $\langle a \rangle \vee \langle b\rangle  := \left\langle \,\langle a \rangle \cup \langle b \right\rangle \, \rangle = \langle \{ a, b \} \rangle= \langle a \oplus b\rangle = \langle a \vee b\rangle$. 
In particular, finitely generated and principal MV-ideals coincide, and 
the principal MV-ideals of $A$ form a sublattice of the lattice of all MV-ideals of $A$.
\end{prop}
\begin{proof}The first assertion is {\cite[Lem. 1.2.1]{cdm}}, and the remaining ones follow from it through a straightforward verification.
\end{proof}

We now recall a few basic facts about the interaction between the MV-algebraic operations and the order.
\begin{lem}\label{lem:minusplusadjunction} In the following,  $a$, $b$, and $c$ are arbitrary elements of the MV-algebra $A$.
\begin{enumerate}
\item The operation $\ominus$ is lower adjoint to $\oplus$, i.e.\ $a\ominus b \leq c \,\Longleftrightarrow \, a\leq b\oplus c$.\, 
\item The operation $\oplus$ preserves all existing meets in each coordinate. That is, for any set $B\subseteq A$
such that $\bigwedge B$ exists, $\bigwedge_{b \in B} (b\oplus a)$ also exists, and we have $\bigwedge_{b \in B}(b\oplus a) = (\bigwedge B)\oplus a$.
Similarly for the second coordinate. In particular, $\oplus$ is order-preserving in each coordinate.
\item The operation $\ominus$ preserves all existing joins in the first coordinate, and reverses all existing joins to meets in the second coordinate. The latter means that for any set $B\subseteq A$
such that $\bigvee B$ exists, $\bigwedge_{b \in B}(a\ominus b)$ exists and we have $a \ominus (\bigvee_{b \in B} b)=\bigwedge_{b \in B}(a\ominus b)$. In particular, $\ominus$ is order-preserving in the first, and order-reversing in the second coordinate.

\item The operation $\oplus$  is join-preserving in the second coordinate, i.e.\  $a \oplus (b\vee c)=(a\oplus b)\vee (a\oplus c)$. Hence, by commutativity, $\oplus$ is join-preserving in each coordinate.
\item The operation $\ominus$  is meet-preserving in the first coordinate, i.e.\  $(a \wedge b)\ominus c=(a\ominus c)\wedge (b\ominus c)$, and 
   meet-reversing in the second coordinate, i.e.\  $a \ominus (b\wedge c)=(a\ominus b)\vee (a\ominus c)$.
\item The equation $(a \ominus b) \wedge (b \ominus a)  = 0$ holds.
\item The equation\footnote{This is the De Morgan-Kleene equation, see \cite[Chapt. XI]{BD1974}.}  $(a\wedge\neg a) \vee (b\vee\neg b)  = b\vee\neg b$ holds.
\item The map $\neg \colon A \to A$ is an order-reversing bijection that is its own inverse.
\end{enumerate}
\end{lem}
\begin{proof} Item 1 is \cite[Lem. 1.1.4(iii)]{cdm}. Items 2--3  are immediate consequences of the adjunction in 1. Items 4--5 are proved as in \cite[Prop. 1.1.6]{cdm}, \textit{mutatis mutandis}. Item 6 is \cite[Prop. 1.1.7]{cdm}. 
For item 7, we need to show that $a \wedge \neg a \leq b \vee \neg b$. By item 1, we have $a \leq (a \ominus b) \oplus b$ and $\neg a \leq (b \ominus a) \oplus \neg b$, using the obvious equality $\neg a \oplus \neg b = b \ominus a$. Using items 2 and 6, we now get
\[ 
a \wedge \neg a \leq ((a \ominus b) \oplus b) \wedge ((b\ominus a) \oplus \neg b) 
\leq [(a \ominus b) \wedge (b \ominus a)] \oplus (b \vee \neg b) = b \vee \neg b.
\]
Item 8 is \cite[Lem. 1.1.3 and Lem. 1.1.4(i)]{cdm}.
\end{proof}

As in Section~\ref{s:stone-priestley}, write $(X,\tau^\downarrow)$ for the dual spectral space of  (the underlying distributive lattice of) $A$,  $Y\subseteq X$ for the subset corresponding to the prime MV-ideals of $A$, and $Z\subseteq Y$ for the subset corresponding to the maximal MV-ideals of $A$. Also Recall from Section \ref{s:stone-priestley} that $X$ is partially ordered by the relation $\leq$ given by inclusion of ideals. The following holds in $Y$, though not in $X$:
\begin{prop}[Cf.\ {\cite[Prop. 10.1.11]{BKW1977}}]\label{prop:separation}
If $y, y' \in Y$, $y \nleq y'$ and $y' \nleq y$, then there exist $u,v \in A$ such that $y \in \hat{u}$, $y' \in \hat{v}$ and $u \wedge v = 0$.
\end{prop}
\begin{proof}
 Pick $a \in I_{y'} \setminus I_{y}$ and $b \in I_{y} \setminus I_{y'}$. Define $u := a \ominus b$ and $v := b \ominus a$. If we had $u \in I_y$, then, since $a \leq (a \ominus b) \oplus b=a\vee b$ by \cite[Prop. 1.1.5]{cdm}, we would get $a \in I_y$, contradicting the choice of $a$. Therefore, $u \not\in I_y$, i.e.\ $y \in \hat{u}$. The proof that $y' \in \hat{v}$ is similar. Since the equation $(a \ominus b) \wedge (b \ominus a) = 0$ holds in any MV-algebra (Lemma~\ref{lem:minusplusadjunction}), we have $u \wedge v = 0$.
\end{proof}
\noindent As a  consequence, we now show that the set of prime MV-ideals of $A$, ordered by inclusion, is a \emph{root system}, i.e.\ a poset such that the upset of any one of its elements is totally ordered. This terminology  first arose in the context of lattice-ordered groups, cf.\ e.g.\  \cite{BKW1977}.
\begin{cor}\label{cor:rootsystem}For each $y\in Y$, ${\uparrow}y$ is totally ordered by $\leq$.  Moreover,
every prime MV-ideal is contained in a unique maximal MV-ideal.\end{cor}
\begin{proof} This is \cite[Thm. 1.2.11(ii) and Cor. 1.2.12]{cdm}. To prove the first assertion,
let $y \in Y$ and suppose that two distinct elements $y', y'' \in Y$ lie in ${\uparrow}y$. Suppose further, by way of contradiction, that $y'$ and $y''$ are incomparable.  By Proposition~\ref{prop:separation}, there are $u \not\in I_{y'}$ and $v \not\in I_{y''}$ such that $u \wedge v = 0$. But then $u \in I_y$ or $v \in I_y$ since $I_y$ is prime, contradicting either $y \leq y'$ or $y \leq y''$.
The second assertion now follows at once by a standard application of Zorn's lemma to the chain ${\uparrow} y$.\end{proof}
We now discuss the topologies inherited by $Y$ from $X$. We show that the subspace  $(Y,\tau^\downarrow)$ is spectral, and that  the distributive lattice dual to $(Y,\tau^\downarrow)$ is isomorphic to the lattice $\mathrm{KCon}\, A$ of principal MV-congruences (equivalently, principal MV-ideals) of $A$.\footnote{There are several places in the literature on MV-algebras and lattice-ordered groups where this result---essentially, Proposition~\ref{prop:YasKPrin}---appears under various guises. The reference closest to the spirit of the present paper is \cite{Cignolietal91}, where the authors interpret via Priestley duality Belluce's results on the map $\lambda$ in \cite{Bel1986}. Compare also \cite[Def. 8.1, Prop. 8.2, and Cor. 8.9]{DuPo2010}.} For there is a lattice homomorphism $\lambda \colon  A \to \mathrm{Con}{\,A}$ from $A$ to the lattice of MV-congruences on $A$ which sends an element $a$ to the MV-congruence $\lambda(a)$ generated by the pair $(a,0)$, or, equivalently, to the principal MV-ideal $\langle a\rangle\subseteq A$.  (Indeed, $\lambda$ clearly preserves $0$ and $1$;  it preserves $\wedge$ and $\vee$ by Proposition \ref{prop:idealgenerated}.) The kernel of $\lambda$ is the congruence $\sigma$ on $A$ such that  $a \, \sigma \, b$ if, and only if, $\langle a \rangle = \langle b\rangle$. The image of the homomorphism $\lambda$ is the lattice $\mathrm{KCon}\,A$ of principal (or, equivalently, finitely generated or compact) MV-congruences of $A$. Hence there is an isomorphism of distributive lattices $A/\sigma \isom \mathrm{KCon}\,A$. Recall that $(X(A/\sigma),\tau^\downarrow)$ denotes the dual spectral space  of  $A/\sigma$.
\begin{prop}\label{prop:YasKPrin}
For any MV-algebra $A$, the spectral space $(X(A/\sigma),\tau^\downarrow)$ dual to the distributive lattice $A/\sigma$ is homeomorphic to the subspace $(Y,\tau^\downarrow)$ of $(X,\tau^\downarrow)$.
\end{prop}
\begin{proof}
By Proposition~\ref{prop:quotientsubspace}, $(X(A/\sigma),\tau^\downarrow)$ is homeomorphic to the subspace
\[ S_\sigma = \{ x \in X \ \mid \ \forall (a,b) \in \sigma \colon (a \in I_x \liff b \in I_x)\}.\]
It thus suffices to show that $S_\sigma = Y$. Let $x \in X$. Suppose first that $x \in Y$, i.e.\ $I_x$ is an MV-ideal. If $(a,b) \in \sigma$ and $a \in I_x$, then $b$ is in the MV-ideal generated by $a$, which is contained in $I_x$. Hence, $b \in I_x$. The proof that $b \in I_x$ implies $a \in I_x$ is symmetric. Hence, $x \in S_\sigma$. Conversely, suppose that $x \in S_\sigma$. We show that $I_x$ is an MV-ideal. If $a, b \in I_x$, then $a \vee b \in I_x$ since $I_x$ is a lattice ideal. By Proposition~\ref{prop:idealgenerated} we have $\langle a \vee b\rangle = \langle a \oplus b\rangle$, i.e.\ $(a \vee b, a\oplus b) \in \sigma$. Since $x \in S_\sigma$, we deduce $a \oplus b \in I_x$, as required.
\end{proof}
\begin{rem}\label{r:mv-spectra}
 The sets $Y$ and $Z$  can also be directly topologized using the MV-ideals of
$A$. The most common such topology on $Y$ is known as the \emph{spectral}, \cite[Def. 4.14]{Mun2011}, traditionally also called the  \emph{hull-kernel topology} when restricted to $Z$: its closed sets are precisely the collections of prime MV-ideals of the form 
\[
\hat{I}^{\, c}:=\{y \in Y \mid I \subseteq I_{y} \}
\]
as $I$ ranges over all  MV-ideals of $A$. This topology is equal to the topology that is induced on $Y$ viewed as a subspace of the spectral space $(X,\tau^\downarrow)$, and thus, by the previous proposition, to the spectral topology on the dual space of the distributive lattice $A/\sigma$. One can similarly define a \emph{dual spectral} topology and a \emph{Priestley topology} on $Y$ by analogy with the lattice case. It is elementary that each of the spectral, dual spectral, and Priestley topologies thus defined on $Y$  agrees with the restriction to $Y$ of  the homonymous topology defined on $X$ as in Section \ref{s:stone-priestley} via the underlying lattice of $A$. \emph{Hence, throughout the paper, we only need consider the restrictions of $\tau^{\downarrow}$, $\tau^{\uparrow}$, and $(\tau^{p},\leq)$ from $X$ to $Y$ and $Z$.} \qed
\end{rem}
\begin{rem}\label{r:filters}As in the case of lattices,  \emph{MV-filters} --- upsets containing $1$ and closed under $a \odot b:= \neg( \neg a \oplus \neg b)$ --- are dual to MV-ideals. Because MV-negation is a dual order-automorphism of the underlying lattice of the MV-algebra $A$, the map $I \mapsto \neg I$, where $\neg I := \{\neg a \mid a \in I\}$, is a bijection between the lattice ideals and filters. It restricts to a bijection between prime or maximal MV-ideals and prime or maximal MV-filters, respectively. 
Note, however, that the lattice-theoretic  bijection $I \longmapsto I^{c}$ between the prime filters and the prime ideals of the underlying distributive lattice of $A$ does not restrict to a bijection between prime MV-ideals and prime MV-filters. Indeed, if $I$ is, say, a maximal MV-ideal of $A$, then simple examples show that $I^{c}$ is not in general an MV-filter of $A$ (though it is, of course, a prime filter of the underlying lattice). Following tradition, we use MV-ideals rather than MV-filters in the sequel. In particular, we stress that $Y$ denotes the set of points of $X$ such that $I_y$ is a prime MV-ideal; this is not the same set as the set of points $Y'$ such that $F_y$ is a prime MV-filter, even though the two are connected via the natural bijection $\neg$ given above. Cf.\ the notation adopted at the beginning of Section \ref{s:mvdual} below.\qed
\end{rem}

\section{Lifting operations and inequalities to the canonical extension} \label{sec:canext}
In the next two sections we study the operations $\oplus$ and $\neg$ on an MV-algebra as additional operations on the underlying distributive lattice. This is done through the use of canonical extensions, which were introduced for Boolean algebras first in \cite{JT1952}, and later generalized to distributive lattices in \cite{GehJon1994}. In this section we collect the relevant definitions and facts from the theory of canonical extensions of distributive lattices. For further background  we refer the reader to \cite{GeJo2004, GePr07a}.

Let $D$ be a distributive lattice with dual Priestley space $(X(D), \tau^p, \leq)$. By the Stone-Priestley Duality Theorem, $D$ is isomorphic to the lattice of clopen downsets of $X(D)$. In particular, $D$ embeds into the complete lattice of all downsets of the poset $(X(D),\leq)$. The canonical extension of $D$ provides a purely algebraic description of this embedding.
\begin{dfn}\label{d:canext}
Let $D$ be a distributive lattice. An embedding $\eta \colon  D \embeds C$, where $C$ is a complete lattice, is called a {\it canonical extension} of $D$ if the following two properties hold.
\begin{enumerate}
\item The image $\eta[D]$ of $D$ is both $\bigvee\,\bigwedge$-dense and $\bigwedge\,\bigvee$-dense in $C$; that is, any element of $C$ is equal to a join of meets and to a meet of joins of elements from $\eta[D]$.
\item If $S$ and $T$ are subsets of $D$ such that $\bigwedge \eta[S] \leq \bigvee \eta[T]$ in $C$, then there are finite subsets $S' \subseteq S$ and $T' \subseteq T$ such that $\bigwedge S' \leq \bigvee T'$ in $D$.
\end{enumerate}
The first condition is commonly referred to as {\it denseness}, and the second condition as {\it compactness}.
\end{dfn}
\begin{thm}\label{thm:canextexun}
Let $D$ be a distributive lattice. The following hold.
\begin{enumerate}
\item \textup{(Existence.\textup)} There exists a canonical extension $\eta : D \embeds D^\delta$  of $D$.
\item \textup{(Uniqueness.)} If $\eta : D \embeds C$ and $\eta' : D \embeds C'$ are canonical extensions of $D$, then there is an isomorphism $\phi : C \isom C'$ such that $\phi \circ \eta = \eta'$.
\end{enumerate}
\end{thm}
\begin{proof}
For  existence, one may show (using Stone's Prime Filter Theorem) that the embedding of $D$ into the complete lattice of downsets of $(X(D),\leq)$ given by $a \mapsto \hat{a}$ is a canonical extension of $D$; see \cite{GehJon1994}. For a construction avoiding non-constructive axioms, see\ e.g.\ \cite[Prop. 2.6]{GehHar2001}. For uniqueness, see \cite[Thm. 2.6]{GehJon1994}.
\end{proof}
We denote the (essentially unique) canonical extension of  $D$ by $D^\delta$;  we identify $D$ with its isomorphic image $\eta[D]\subseteq D^\delta$. By the uniqueness statement in the theorem it is easy to verify that for any distributive lattices $D$ and $E$, one has $(D \times E)^\delta \isom D^\delta \times E^\delta$ and $(D^\dual)^\delta \isom (D^\delta)^\dual$, where $(\cdot)^{\dual}$ denotes the lattice in the opposite order.

 An element $x \in D^\delta$ is  a {\it filter} (or {\it closed}) {\it element} if it is a meet of elements of $D$. Order-dually, $y \in D^\delta$ is an {\it ideal} (or {\it open}) {\it element} if it is a join of elements of $D$. Any lattice filter $F$ of $D$ gives rise to the filter element $\bigwedge F \in D^\delta$, and any lattice ideal $I$ of $D$ gives rise to the ideal element $\bigvee I \in D^\delta$. The compactness property of $D^\delta$ is equivalent to the statement that for any lattice filter $F$ and lattice ideal $I$ of $D$, one has $\bigwedge F \leq \bigvee I$ if, and only if, $F \cap I \neq \emptyset$. 
The assignment 
\begin{align}
\label{eq:filterbijection}
F \longmapsto \bigwedge F
\end{align}
yields an isomorphism between the complete lattice of filters of $D$, ordered by reverse inclusion, and the complete lattice of filter elements of $D^\delta$, that we denote by $F(D^\delta)$. Order-dually, 
\begin{align}
\label{eq:idealbijection}
I \longmapsto \bigvee I
\end{align}
is an isomorphism between ideals of $D$, ordered by inclusion, and the ideal elements of $D^\delta$, denoted $I(D^\delta)$.

We denote the set of completely join-irreducible elements\footnote{
Recall that $x \in D$ is \emph{completely join-irreducible} if, whenever $x = \bigvee S$ for some  subset $S \subseteq D$, we have $x \in S$. Note that $0 = \bigvee \emptyset$ is never completely join-irreducible. The definition of {\it completely meet-irreducible} is dual, and $1$ is never completely meet-irreducible since $1 = \bigwedge \emptyset$.}
of $D^\delta$ by $J^\infty(D^\delta)$, or just $J^\infty$ when $D^{\delta}$ is clear from the context, and the set of completely meet-irreducible elements of $D^\delta$ by $M^\infty(D^\delta)$ or $M^\infty$. Observe that $J^\infty(D^\delta) \subseteq F(D^\delta)$ and $M^\infty(D^\delta) \subseteq I(D^\delta)$, by the denseness property. It can be shown that the assignment $F \longmapsto \bigwedge F$ restricts to a bijection between the set of prime filters of $D$ and the set $J^\infty(D^\delta)$; order-dually, the assignment $I \longmapsto \bigvee I$ restricts to a bijection between the set of prime ideals of $D$ and the set $M^\infty(D^\delta)$. Under these identifications, the bijection $F\longmapsto F^{c}$ between the sets of prime filters and prime ideals of $D$ corresponds to the bijection
\begin{align}
\label{eq:kappadef}
\kappa \colon J^\infty(D^\delta)\ &\longrightarrow \ M^\infty(D^\delta)\\  x \quad \quad &\overset{\kappa}{\longmapsto} \ \bigvee \{ a \in D \mid x \nleq a \}\nonumber
\end{align}
which has the important property
\begin{align} 
\label{eq:kappaprop}
\forall u \in D^\delta,\, x \in J^\infty(D^\delta) \ : \ x \leq u \iff u \nleq \kappa(x)\,.
\end{align}
We  remark that, under the isomorphisms $(D^\delta)^n \isom (D^n)^\delta$ and $(D^\delta)^\dual \isom (D^\dual)^\delta$, where $n\geq 1$ is an integer, we have $F((D^\delta)^n) \isom F(D^\delta)^n$ and $F((D^\delta)^\dual) \isom I(D^\delta)$;  order-dual statements hold for ideal elements.

Let $f \colon D^n \to D$ be an order-preserving $n$-ary operation on a distributive lattice $D$. We define maps $f^\sigma, f^\pi \colon (D^\delta)^n \to D^\delta$. For a filter element $x$ of $(D^\delta)^n$, let 
\begin{align}\label{eq:lift1}
\overline{f}(x) &:= \bigwedge\{f(a) \ | \ x \leq a \in D^n\}, 
\end{align}
and for an ideal element $y$ of $(D^\delta)^n$, let 
\begin{align}\label{eq:lift2}
\overline{f}(y) &:= \bigvee \{f(a) \ | \ y \geq a \in D^n\}.
\end{align}
Now, for $u \in (D^\delta)^n$, we define
\begin{align} 
f^\sigma(u) &:= \bigvee \{\overline{f}(x) \ | \ u \geq x \in F(D^\delta)^n\}, \label{eq:sigma}\\
f^\pi(u) &:= \bigwedge \{\overline{f}(y) \ | \ u \leq y \in I(D^\delta)^n \}.\label{eq:pi}
\end{align}
The operations (\ref{eq:sigma}--\ref{eq:pi}) are called the {\it  $\sigma$-extension} and the {\it  $\pi$-extension}  of $f$, respectively.\footnote{For a more general definition of $\sigma$- and $\pi$-extensions that does not assume monotonicity of $f$, see \cite[Sec. 2.4]{GeJo2004}.}
Although they both restrict to $\overline{f}$ on filter and ideal elements of $D^\delta$,  they do not necessarily coincide.\footnote{The operation $\oplus$ of Chang's non-simple totally ordered MV-algebra, for example, has distinct  $\sigma$- and $\pi$-extensions; see \cite[Prop. 1]{GePr00}.} 
The above  is easily  adapted to operations which are order-reversing in some coordinates. If $g \colon D \times D \to D$ is, say,  order-preserving in the first coordinate and order-reversing in the second coordinate, then $g$ acts as an order-preserving operation on $D \times D^\dual$, so the preceding definitions  apply up to the appropriate order flips. 

It will be important that residuated (=adjoint) pairs of operations lift to the canonical extension:\footnote{This was first proved, in the distributive case, by B.\ J\'onsson during early work in 1996 on the paper \cite{GeJo2004}.  Here we  give a new proof based on a general method from \cite{GNV2005}. See also \cite[Prop. 2]{Ge2012} for a different  proof, in the context of Heyting algebras.}
\begin{prop} \label{prop:adjunctionlifts}
Let $D$ be a distributive lattice, and suppose that $f \colon D \times D \to D$ and $g \colon D \times D^\dual \to D$ are order-preserving operations such that 
\[ \forall a, b, c \in D : f(a,b) \leq c \ \iff \  a \leq g(c,b).\] 
Then:
\[ \forall u, v, w \in D^\delta: f^\sigma(u,v) \leq w \ \iff  \ u \leq g^\pi(w,v).\] 
\end{prop}
\begin{proof}
We only prove that $f^\sigma(u,v) \leq w \Rightarrow u \leq g^\pi(w,v)$, the other direction being similar. Let us write $E := D \times D \times D^\dual$, and define operations $p, s, t : E \to D$ by
\begin{align*} 
p(a,b,c) &:= a, \quad s(a,b,c) := g(c,b), \\
t(a,b,c) &:= 0 \text{ if } f(a,b) \leq c, \text{ and } t(a,b,c) := 1 \text{ if } f(a,b) \nleq c.
\end{align*}
Observe that, using the assumption, the inequality $p \leq s \vee t$ holds pointwise on $E$. Therefore, $p^\sigma \leq (s \vee t)^\sigma$ holds pointwise on $E^\delta$. Note that $t$ is order-preserving (as a function from $E$ to $D$), while $s$ is order-reversing. Hence, by \cite[Lem. 5.11]{GNV2005}, we have $(s \lor t)^\sigma \leq s^\pi \lor t^\sigma$, so $p^\sigma \leq s^\pi \lor t^\sigma$.
Now, it is not hard to see from the definitions of the extensions that, for all $u, v, w \in D^\delta$, we have  $p^\sigma(u,v,w) = u$, $s^\pi(u,v,w) = g^\pi(w,v)$, and $t^\sigma(u,v,w) = 0$ if $f^\sigma(u,v) \leq w$. In particular, if $f^\sigma(u,v) \leq w$, then we get
\[ u = p^\sigma(u,v,w) \leq (s^\pi \lor t^\sigma)(u,v,w) = g^\pi(w,v) \lor 0 = g^\pi(w,v). \qedhere\]
\end{proof}
Equally important will be that certain inequalities also lift to the canonical extension. By a {\it binary dual operator} we mean a binary operation on $D$ which preserves finite meets in both coordinates. The following proposition is a special case of {\cite[Thm. 4.6]{GehJon1994}}, where the result is proved for operations which in each coordinate preserve either finite joins or finite meets.
\begin{prop}\label{thm:canonicity}
Let $s$ and $t$ be terms in the language of distributive lattices with an additional binary function symbol $f$. For any distributive lattice $D$, and for any binary dual operator $f^D$  on $D$, if $(D,f^D)$ satisfies the inequality $s \leq t$ then so does $(D^\delta, (f^D)^\pi)$.
\end{prop}
\section{The structure of the lattice spectrum of an MV-algebra}\label{s:mvdual}

In this section and the next we examine the structure of the lattice spectrum of an MV-algebra. In particular, here we show that the operations $\oplus$ and $\neg$ of an MV-algebra yield dual operations on the lattice spectrum which make it into a topological partial commutative semigroup with an involution (Propositions~\ref{prop:plusprops} and \ref{prop:kleeneinvolution}).
\begin{notation}We  henceforth adopt the   conventions below.
\begin{itemize}
\item $A$ denotes an MV-algebra, and its bounded lattice reduct, assuming it is clear from the context which is meant.
\item $X$ denotes the  set of points of (the underlying distributive lattice of) $A$, in the sense of  Section \ref{s:stone-priestley}. To each point $x\in X$ there is associated a prime ideal $I_x$, and a prime filter $F_x$.  
\item $Y\subseteq X$ denotes the set of points $y \in X$ such that $I_y$ is a prime MV-ideal of $A$.
\item $Z\subseteq Y$ denotes the set of points $z \in Y$ such that $I_z$ is a maximal MV-ideal of $A$.
\item The partial order $\leq$ on $X$ and on its subsets $Y$ and $Z$ is defined as the inclusion order of the corresponding ideals, as in Section \ref{s:stone-priestley}.
\item $X$ is equipped either with the spectral topology $\tau^\downarrow$, or with the co-spectral topology $\tau^\uparrow$, or with the Priestley  topology $\tau^{p}$ and with the inclusion order $\leq$ of the ideals $I_x$, $x \in X$. We will often want to make it clear which topology on $X$ is meant, and thus accordingly write $(X,\tau^\downarrow)$, $(X,\tau^\uparrow)$ and $(X,\tau^p,\leq)$, and similarly for the topologies on the sets $Y$ and $Z$.
\qed
\end{itemize}
\end{notation}

Recall from Remark \ref{r:filters} that the map which sends a (prime) MV-ideal $I$ to the (prime) MV-filter $\neg I$ is a bijection. In terms of canonical extensions, this map is the unique extension $(\neg)^\sigma = (\neg)^\pi$ of the operation $\neg$. For $x \in X$, we can thus define $i\colon X \to X$ by declaring $i(x)$ to be the unique element of $X$ with
 \begin{align}\label{def:i}
I_{i(x)} = \neg F_x, \text{ or equivalently, } F_{i(x)} = \neg I_x.
 \end{align}
Observe that the subspace $Y$ is not closed under the operation $i$: for $y \in Y$, $I_{i(y)}$ is the complement of the prime MV-filter $\neg I_y$, which is a prime ideal, but not necessarily an MV-ideal.
\begin{prop}[Cf. {\cite[Thm. 3.2]{CorFow1977}}]\label{prop:kleeneinvolution}
The map $i \colon X \to X$ is an order-reversing homeomorphism of the Priestley space $(X,\tau^{p},\leq)$ which is its own inverse. Moreover, for all $x \in X$, either $x \leq i(x)$ or $i(x) \leq x$.
\end{prop}
\begin{proof}
Both assertions follow at once from the Stone-Priestley Duality Theorem, using  8 and 7 in Lemma \ref{lem:minusplusadjunction}, respectively. See also \cite[Subsec. 6.3.4]{GNV2005}.
\end{proof}
Next consider the  lift $\iplus$ of the operation $\oplus$ to ideal elements of $A^\delta$, as in (\ref{eq:lift2}). We have:
\[ \left(\bigvee I\right) \iplus \left(\bigvee J\right) = \bigvee \{ a \oplus b \ | \ a \in I, b \in J \} = \bigvee \{ c \ | \ \exists a \in I, b \in J : c \leq a \oplus b\}\,.\]
Hence we define $\iplus$ on ideals $I, J\subseteq A$ as follows:\footnote{An operation which is essentially the restriction of (\ref{eq:iplusdef}) to prime  ideals was considered in the context of Wajsberg algebras in \cite{Ma1990}.}
\begin{align}
\label{eq:iplusdef}
I \iplus J &:= \{c \in A \ | \ \exists a \in I, b \in J \text{ such that } c \leq a \oplus b\}.
\end{align}
Similarly, we lift $\ominus$ to filter elements of $(A \times A^\dual)^\delta$.
For $F$ a lattice filter and $I$ a lattice ideal of $A$, define the lattice filter
\begin{align}
\label{eq:iminusdef}
F \iminus I := \{c \in A \ | \ \exists a \in F, b \in I \text{ such that } c \geq a \ominus b\}.
\end{align}
\begin{rem}
Let us emphasize  that the operations $\overline{\oplus}$ and $\overline{\ominus}$ above agree with the lifted operations $\overline{\oplus}$ and $\overline{\ominus}$ on filter and ideal elements of the canonical extension, defined in the preceding Section \ref{sec:canext}, upon using the isomorphisms between the filter elements with lattice filters and the ideal elements with lattice ideals, (\ref{eq:filterbijection}--\ref{eq:idealbijection}). Hence we can apply the methods from Section \ref{sec:canext} and the papers \cite{GePr07a,GePr07b}, as we will now do.\qed
\end{rem}

\begin{prop}\label{prop:additionproperties}
Let $I$ and $J$ be  ideals of the MV-algebra $A$, and let $F$ be a filter of $A$. 
\begin{enumerate}
\item The set $I \iplus J$ is the  ideal generated by the elements $a \oplus b$, for $a \in I$ and $b \in J$.
\item The operation $\iplus$ is commutative, associative, and has the ideal $\{0\}$ as a neutral element.
\item The operation $\iplus$ is monotone: if $J \subseteq J'$, then $I \iplus J \subseteq I \iplus J'$.
\item The  ideal $I \iplus J$ contains $1$ if, and only if, there exists $a \in I$ such that $\neg a \in J$.
\item The  ideal $I$ is an MV-ideal if, and only if, $I \iplus I \subseteq I$.
\item  The set $F \iminus I$ is a filter, and
\[
F \iminus I \subseteq J^c \iff F \subseteq (J \iplus I)^c.
\]
\end{enumerate}
\end{prop}
\begin{proof}
The second item is a consequence of Proposition \ref{thm:canonicity}, upon noting that the $\pi$-extension $\oplus^{\pi}$ of $\oplus$ (see (\ref{eq:pi})) is a binary dual operator by Lemma \ref{lem:minusplusadjunction} and Proposition \ref{prop:adjunctionlifts}.  Items 1, 3, 4, and 5 are immediate from the definitions.
For item 6, apply Proposition \ref{prop:adjunctionlifts} to the   operations $f := \ominus$ and $g := \oplus$, which  form an adjoint pair by Lemma \ref{lem:minusplusadjunction}.1. Then:
\[ \left(\bigwedge F\right) \iminus \left(\bigvee I\right) \leq \bigvee J \iff \bigwedge F \leq \left(\bigvee J\right) \iplus \left(\bigvee I\right).\]
Since $\bigwedge F \leq \bigvee I$ iff $F \cap I \neq \emptyset$ by the compactness property of Definition \ref{d:canext}, a straightforward rewriting completes the proof.
\end{proof}
The following proposition is central to the paper \cite{GePr07a}; cf. Lemma 4.3 and Theorem 4.4 therein.
\begin{prop}\label{prop:primeplusidealcanext}
Let $A$ be an MV-algebra. The following hold:
\begin{enumerate}
\item For $x \in M^\infty(A^\delta)$, $y \in I(A^\delta)$, we have $x \iplus y \in M^\infty(A^\delta) \cup \{1\}$. 
\item For $j \in J^\infty(A^\delta)$, $x \in M^\infty(A^\delta)$, we have $j \iminus x \in J^\infty(A^\delta) \cup \{0\}$.
\end{enumerate}
\end{prop}
\begin{proof}Adopting the terminology of \cite{GePr07a},  say that an operation $h \colon A \times A \to A$ is a {\it double operator} if it preserves binary joins and binary meets in each coordinate. By Lemma~\ref{lem:minusplusadjunction}, $\oplus\colon A \times A \to A$ and  $\ominus \colon  A \times A^\dual \to A$ are double operators.
For the first item, set $h(a,b) := b \ominus a$. Then  $h(a,0) = 0$ for all $a \in A$. Writing $C := A^\delta$, the right upper adjoint of $h^\sigma$ is the operation $l \colon C \times C \to C$, which sends $(u,v)$ to $v \oplus^\pi u$, by Proposition~\ref{prop:adjunctionlifts}. By the proof of \cite[Thm. 4.4]{GePr07a}, we get in particular that $l$ maps an element $(y,x) \in F(C^\dual) \times M^\infty(C) = I(C) \times M^\infty(C)$ into $M^\infty(C) \cup \{1\}$, so that indeed $l(y,x) = x \oplus^\pi y \in M^\infty(A^\delta) \cup \{1\}$, as required. The proof of the second item is dual, and uses that $\oplus$ is a double operator of which $\ominus$ is a lower adjoint (Lemma \ref{lem:minusplusadjunction}).
\end{proof}
\begin{cor}\label{cor:primeplusideal}
Let $x \in X$ and let $J$ be an ideal of $A$. Then $I_x \iplus J$ is a  prime ideal of $A$ if, and only if, $I_{i(x)} \supseteq J$. In particular,  given $x, y \in X$,  the ideal $I_x \iplus I_y$ of $A$ is prime if, and only if, $i(x) \geq y$.
\end{cor}
\begin{proof}
By Proposition~\ref{prop:primeplusidealcanext}, if $1 \not\in I_x \iplus J$, then $I_x \iplus J$ is a  prime ideal. By  Proposition~\ref{prop:additionproperties}.4, we have  $1 \not\in I_x \iplus J$ if, and only if, there is no $a \in I_x$ such that $\neg a \in J$  iff $\overline{(\neg)}I_x \subseteq J^c$. This is equivalent to $I_{i(x)} \supseteq J$ by the definition (\ref{def:i}) of $i(x)$.
\end{proof}
We now define a partial binary operation $+$ on $X$ with domain 
\[\dom(+) := \{(x,y) \in X^2 \ | \ i(x) \geq y \} = \{(x,y) \in X^2 \ | \ 1 \not\in I_x \iplus I_y\}.\]
For $(x,y) \in \dom(+)$, we let $x + y$ be the unique element of $X$ such that 
\begin{align}\label{def:+}
I_{x+y} = I_x \iplus I_y.
\end{align}
Note that, by Proposition \ref{prop:additionproperties}, $1\not\in I_y \iplus I_y = I_y$ for $y \in Y$, so that $i(y) \geq y$ for all $y \in Y$. Hence $I_{y}+I_{y}$ is defined for all $y \in Y$.

\begin{rem}\label{r:inotneeded} We observe in passing that the dual  $i\colon X \to X$ of negation as in (\ref{def:i}) is definable from  $+\colon \dom(+)\to X$. Indeed, one  has %
\[
i(x) = \max{ \{ y \in X \ | \ (x,y) \in \dom(+)\}}.
\]
Strictly speaking, therefore, carrying along the structure $i\colon X \to X$ on the dual space is not necessary, but it will turn out to be convenient to have.\qed
\end{rem}

Importantly, we obtain:
\begin{prop}[Dual structure]\label{prop:plusprops}
The structure $(X, \tau^\uparrow, +)$ is a topological partial commutative semigroup that is translation-invariant with respect to the specialization order, and whose set of idempotent elements is the MV-spectrum $Y$ of $A$. More precisely, for all $x, x', x'' \in X$, the following hold.
\begin{enumerate}
\item \textup{(Commutativity.)} If $x + x'$ is defined, then $x' + x$ is defined, and $x + x' = x' + x$.
\item \textup{(Associativity.)} If $x + x'$ and $(x + x') + x''$ are defined, then $x + (x' + x'')$ is defined, and $(x + x') + x'' = x + (x' + x'')$.
\item \textup{(Translation-invariance.)} If $x' \leq x''$ and $x + x''$ is defined, then $x + x'$ is defined, and $x + x' \leq x + x''$.
\item \textup{(Idempotents are Priestley-closed.)} The ideal $I_x$ is a prime MV-ideal if, and only if, $x + x$ is defined and $x + x = x$. Hence 
\begin{align*}
Y&=\{y \in X \mid (y,y) \in \dom(+) \text{\textup{ and }}  y + y \leq y\}\\ &=\{y \in X \mid (y,y) \in \dom(+)  \text{\textup{ and }} y + y = y\},
\end{align*}
 and $Y$ is closed in $(X,\tau^{p})$.
\item  \textup{(Continuity.)} The function $+ : \dom(+) \to X$ is continuous with respect to the topology $\tau^\uparrow$ on $X$ and the subspace topology inherited by $\dom(+)$ from the product space $(X,\tau^\uparrow)\times (X,\tau^\uparrow)$.
\item \textup{(Closed domain.)} The set $\dom(+)\subseteq X^{2}$ is closed in the product space $(X,\tau^\uparrow) \times (X,\tau^\uparrow)$.
\end{enumerate}
\end{prop}
\begin{proof}
The first three items follow directly from Proposition~\ref{prop:additionproperties}. 
For item 4, suppose $I_x$ is a prime MV-ideal. Then $I_x \iplus I_x \subseteq I_x$ by  Proposition~\ref{prop:additionproperties}.5, so in particular $x + x$ is defined. It follows that $x + x \leq x$. On the other hand, $x \leq x + x$ always holds since $I_x = I_x \iplus 0 \subseteq I_x \iplus I_x$ by Proposition~\ref{prop:additionproperties}. Thus we infer $x + x = x$. The converse direction is clear from  Proposition~\ref{prop:additionproperties}.5. The fact that $Y$ is closed in $(X, \tau^{p})$ is an immediate consequence of Propositions \ref{prop:YasKPrin} and  \ref{prop:quotientsubspace}.

For item 5, pick $a \in A$. Note that, for $(x,x') \in \dom(+)$, we have $a \in I_{x + x'}$ iff there exist $b \in I_x$ and $c \in I_{x'}$ such that $a \leq b \oplus c$, so
\[ +^{-1}(\hat{a}^{\, c}) = \dom(+) \cap \bigcup \{ \hat{b}^{\, c} \times \hat{c}^{\, c} \ | \ b, c \in A : a \leq b \oplus c \},\]
which is clearly open in $\dom(+) \subseteq (X,\tau^\uparrow)^2$.

For the last item, recall from Section \ref{s:stone-priestley} that $\leq$ is a closed subset of $(X,\tau^\downarrow) \times (X,\tau^\uparrow)$, because $(X,\tau^p,\leq)$ is a Priestley space. Since $i$ is an order-reversing continuous function by Proposition~\ref{prop:kleeneinvolution}, $\dom(+) = \{(x,y) \ | \ i(x) \geq y\}$ is closed in $(X,\tau^\uparrow) \times (X,\tau^\uparrow)$.
\end{proof}
\begin{notation}
Henceforth, given $x,x',x''\in X$, we write  `\,$x + x' = x''$\,', `\,$x' \leq x' + x''$\,', and so forth, to mean `both sides of the (in)equality are defined, and the (in)equality holds'.
\end{notation}
\section{The decomposition of the lattice spectrum: the map $k$}\label{sec:decomp}
In this section we construct the map $k\colon Y \to X$ that is central to our results, and establish some of its order-topological properties, which naturally arises here (and, implicitly, in \cite{GePr07a}) from the canonical extension of an MV-algebra. In particular, we obtain an \'etale-like decomposition of $X$ into simple fibres in Proposition \ref{prop:chainusingk}, and the crucial  Interpolation Lemma \ref{lem:interpolation}. 

By way of extension of Proposition \ref{prop:quotientsubspace}, MV-algebraic quotients of $A$ may be  identified dually using the operation $\iplus$ on ideals defined in (\ref{eq:iplusdef}).
\begin{prop}\label{prop:quotientdually}
Let $J$ be an MV-ideal. The Priestley lattice spectrum of the quotient MV-algebra $A/J$ is homeomorphic to the Priestley-closed subspace of $X$ defined by
\begin{align*} 
S_J := \{ x \in X \ | \ I_x \iplus J \subseteq I_x\}.
\end{align*}
Moreover, if $J = I_y$ for some $y \in Y$, then $S_{I_y} = \{x \in X \ | \ x + y \leq x\}$ is totally ordered.
\end{prop}
\begin{proof}
Let us write $\theta$ for the MV-congruence on $A$ which is the kernel of $A \onto A/J$. Since $A/J$ is in particular a lattice quotient of $A$, we may apply Proposition~\ref{prop:quotientsubspace}: the lattice spectrum of $A/J$ is homeomorphic to the closed subspace 
$$S_\theta = \{x \in X \ | \ \forall (a,b) \in \theta : (a \in I_x \liff b \in I_x)\}$$ 
with the subspace topology and the restricted order. We show that $S_J = S_\theta$.

Suppose that $x \in S_\theta$. Let $a \in I_x$ and $c \in J$. Then $(a, a \oplus c) \in \theta$ since $(0,c) \in \theta$ and $\theta$ is an MV-congruence. Since $x \in S_\theta$ and $a \in I_x$, we get $a \oplus c \in I_x$. So $\{a \oplus c \ | \ a \in I_x, c \in J\} \subseteq I_x$, and therefore the lattice ideal generated by this set is contained in $I_x$, because $I_x$ is a lattice ideal. So $I_x \iplus J \subseteq I_x$, i.e., $x \in S_J$. Conversely, suppose that $x \in S_J$, i.e., $I_x \iplus J \subseteq I_x$. We show that $x \in S_\theta$. Let $(a,b) \in \theta$, and suppose that $a \in I_x$. Since $(a,b) \in \theta$ we have $(b \ominus a,0) \in \theta$, so $b \ominus a \in J$. Hence, $a \vee b = a \oplus (b \ominus a) \in I_x \iplus J \subseteq I_x$, so $a \vee b \in I_x$. Hence we get in particular $b \in I_x$, because $I_x$ is a downset. Since $\theta$ is symmetric, $b \in I_x$ implies $a \in I_x$, so $x \in S_\theta$. We conclude that the lattice spectrum of $A/J$ is indeed homeomorphic to $S_J$. 

To prove the second assertion, recall from Proposition~\ref{prop:primechain} that if $J = I_y$ is prime, then $A/I_y$ is totally ordered. Hence, $S_{I_y}$ is also totally ordered, because by a standard exercise the prime ideals of a chain form a chain. The fact that $S_{I_y} = \{x \in X \ | \ x + y \leq x\}$ is immediate from the definition of $+$.
\end{proof}
\begin{notation}Henceforth, if $y \in Y$, we write $C_y= \{x \in X \ | \ x + y \leq x\}$ for the chain $S_{I_y} \subseteq X$.
\end{notation}
Our next aim is to show that, for a fixed $x \in X$, there exists a \emph{maximum} element $y \in Y$ such that $x + y \leq x$; this will  lead us to the definition of a map $k\colon X \to Y$ (Definition~\ref{def:k}). We prove this through canonical extensions; we use the bijection $\kappa\colon J^\infty(A^\delta)\to M^\infty(A^\delta)$ defined in (\ref{eq:kappadef}). 
\begin{lem}\label{prop:largestprimecanext}
 For any $x \in M^\infty(A^\delta)$, the set $$\{ u \in A^\delta \ | \ x \oplus^\pi u \leq x\}$$ is closed under $\oplus^\pi$, and has a maximum, which is given by $\kappa(\kappa^{-1}(x) \iminus x) \in M^\infty(A^\delta)$.
\end{lem}
\begin{proof} 
To see that the set is closed under $\oplus^\pi$, notice that if $u, v$ are in the set, then 
$$x \oplus^\pi (u \oplus^\pi v) = (x \oplus^\pi u) \oplus^\pi v \leq x \oplus^\pi v \leq x,$$
using that $\oplus^\pi$ is associative by Proposition~\ref{thm:canonicity}, and order-preserving by Proposition~\ref{prop:additionproperties}.3 and the definition of $\oplus^\pi$.

Let $j := \kappa^{-1}(x) \in J^\infty(A^\delta)$. We first use Proposition~\ref{prop:primeplusidealcanext} to show that $j \iminus x \in J^\infty(A^\delta)$. To obtain a contradiction, if we had $j \iminus x = 0$ in $A^\delta$, then by adjunction (Lemma \ref{lem:minusplusadjunction} and Proposition \ref{prop:adjunctionlifts}) and the observation that $0$ is a neutral element for $\iplus$, we would conclude $j \leq x \iplus 0 = x = \kappa(j)$, which contradicts the definition of $\kappa(j)$. So $j \iminus x \in J^\infty(A^\delta)$. We now note, for any $u \in A^\delta$,
\begin{align*}
x \oplus^\pi u \leq x &\iff j \nleq x \oplus^\pi u &(\text{using (\ref{eq:kappaprop}) and $x = \kappa(j)$})\\
&\iff j \iminus x \nleq u &(\text{Proposition~\ref{prop:adjunctionlifts}}) \\
&\iff u \leq \kappa(j \iminus x) &(\text{using (\ref{eq:kappaprop})}).
\end{align*}
Therefore, the maximum of the set $\{ u \in A^\delta \ | \ x \oplus^\pi u \leq x\}$ is $\kappa(j \iminus x)$, which is indeed an element of $M^\infty(A^\delta)$.
\end{proof}

We can now prove:
\begin{prop}\label{prop:kdefinition} \textup{(i)} \,
For any $x \in X$, there exists a largest  ideal $J$ of $A$ such that $I_x \iplus J \subseteq I_x$, namely, the prime  ideal $(F_x \iminus I_x)^c$. \textup{(ii)} \, For any $x \in X$, $(F_x \iminus I_x)^c$ is in fact a prime MV-ideal.
\end{prop}
\begin{proof}
(i) \, Write $k(x) := \kappa(\kappa^{-1}(x) \iminus x)$. By Lemma~\ref{prop:largestprimecanext}, the largest such ideal is $I_{k(x)}$, which is indeed equal to $(F_x \iminus I_x)^c$.

(ii) \, By Proposition~\ref{prop:largestprimecanext}, the set of elements $u$ such that $x \oplus^\pi u \leq x$ is closed under $\oplus^\pi$. In particular, $k(x) + k(x)$ is in this set, so $k(x) + k(x) \leq k(x)$. It now follows that $I_{k(x)}$ is an MV-ideal by Propostion~\ref{prop:plusprops}.
\end{proof}
\begin{dfn}\label{def:k}In light of Proposition \ref{prop:kdefinition}, we define a function $k\colon X \to Y$ by letting $k(x)\in Y$ be such that $I_{k(x)}$ is the largest prime MV-ideal of $A$ satisfying $I_x \iplus J \subseteq I_x$; equivalently, $k(x) := \kappa(\kappa^{-1}(x) \iminus x)$.
\end{dfn}
The fibres of the map $k\colon X \to Y$ decompose $X$ in a  simple manner. 
\begin{prop}[Decomposition by $k$]\label{prop:chainusingk}
For any $y \in Y$, we have $k^{-1}({\uparrow}y) = C_y$.
\end{prop}
\begin{proof}
Let $y \in Y$. For any $x \in X$, we have $k(x) \geq y$ iff $x + y \leq x$ iff $x \in C_y$, by definition of $C_y$.
\end{proof}
\begin{prop}\label{prop:kcontinuity}
The map $k : X \to Y$ has the following properties.
\begin{enumerate}
\item For $x \in X$, we have
\begin{align*}
I_{k(x)} 
         &= \{a \in A \ | \ \forall c \in A : (c \ominus a \in I_x \rightarrow c \in I_x)\}.
\end{align*}        
\item The set $k^{-1}(\hat{a})$ is open in $(X,\tau^p)$ for any $a \in A$. 
\item The map $k\colon (X,\tau^p) \to (Y,\tau^\downarrow)$ is a continuous function.
\end{enumerate}
\end{prop}
\begin{proof}
For item 1,  recall that $k(x) = \kappa(\kappa^{-1}(x) \iminus x)$, when $x$ is regarded as an element of $M^\infty(A^\delta)$. Therefore, 
\begin{align*}
a \not\in I_{k(x)} &\iff a \nleq \kappa(\kappa^{-1}(x) \iminus x) &\text{(using (\ref{eq:idealbijection}))}\\
&\iff \kappa^{-1}(x) \iminus x \leq a &\text{(using (\ref{eq:kappaprop}))}\\
&\iff \kappa^{-1}(x) \iminus a \leq x &\text{(*)}\\
&\iff \bigwedge_{\kappa^{-1}(x) \leq c} \left(c \ominus a\right) \leq x &\text{(def. of $\iminus$ on $F(A^\delta)$)}\\
&\iff \exists c \in A \ : \kappa^{-1}(x) \leq c \text{ and } c \ominus a \leq x &(x \in M^\infty(A^\delta))\\
&\iff \exists c \in A \ : \  c \not\in I_x \text{ and } c \ominus a \in I_x &\text{(using (\ref{eq:idealbijection}))} ,
\end{align*}
where the equivalence marked (*) follows easily from the adjunction  between $\iminus$ and $\iplus$ (Lemma \ref{lem:minusplusadjunction}.1 and Proposition~\ref{prop:adjunctionlifts}), and the fact that $\iplus$ is commutative. From the above chain of equivalences, item 1 is clear. Item 2 follows from item 1 upon noting that $a \in I_{k(x)}$ if, and only if $x \in \bigcap_{c \in A} (\hat{c \ominus a} \cup \hat{c}^c)$, so that
\[ k^{-1}(\hat{a}^c) = \bigcap_{c \in A} (\hat{c \ominus a} \cup \hat{c}^c),\]
which is closed in $(X,\tau^p)$. Item 3 is now immediate from the definition of the topology $\tau^\downarrow$.
\end{proof}
\begin{rem}\label{r:martinez} Proposition~\ref{prop:kcontinuity}.1 shows that our function $k$ is the same as the function $K$ defined in \cite[Thm. 6.1.3]{cdm}, where it is  attributed to the PhD thesis of N. G. Mart\'inez, also see \cite{Ma1996}. \qed
\end{rem}
The fixed points of the map $k$ are exactly the points of $X$ corresponding to prime MV-ideals. That is, $k$ is a retraction of $X$ onto $Y$.
\begin{prop}\label{prop:fixpointsk}
Let $x \in X$. Then $x \in Y$ if, and only if, $k(x) = x$.
\end{prop}
\begin{proof}
For the non-trivial implication, recall from Proposition~\ref{prop:plusprops}.4 that if $x \in Y$, then $x + x$ is defined and $x + x = x$. In particular, by definition of $k(x)$, we have $x \leq k(x)$. On the other hand, using translation-invariance of $+$ (Proposition ~\ref{prop:plusprops}.3), we get $k(x) = 0 + k(x) \leq x + k(x) \leq x$. We conclude that $x = k(x)$.
\end{proof}
Finally, we show how the map $k$ relates to the order of the space $X$. This will be of crucial importance in our applications.
\begin{lem}[Interpolation Lemma]\label{lem:interpolation}
Let $x, x' \in X$ be such that $x \leq x'$. There exists $x'' \in X$ such that $x \leq x'' \leq x'$, $k(x'') \geq k(x)$ and $k(x'') \geq k(x')$.
\end{lem}
\begin{proof}We use the properties of $+$ established in Proposition~\ref{prop:plusprops}.
Let $x'' := x + k(x')$. Note that $x''$ is well-defined, because $k(x') \leq i(x') \leq i(x)$.
Clearly, $x \leq x''$, by monotonicity of $+$. Also, $x'' = x + k(x') \leq x' + k(x') \leq x'$.
To show that $k(x) \leq k(x'')$, it suffices to show that $x'' + k(x) = x''$, by definition of $k(x'')$. We calculate:
\begin{align*} 
x'' + k(x) &= x + k(x') + k(x) \\
&= x + k(x) + k(x') \\
&= x + k(x') = x''.
\end{align*}
Similarly, to show that $k(x') \leq k(x'')$, we prove that $x'' + k(x') = x''$:
\begin{align*}
x'' + k(x') &= x + k(x') + k(x') \\
&= x + k(x') = x''.\qedhere
\end{align*}
\end{proof}
\section{Kaplansky's Theorem for MV-algebras: the map $m$}\label{sec:kaplansky}
In \cite{kaplansky}, Kaplansky  proved that for any compact Hausdorff space $S$  the (unbounded) distributive lattice of all real-valued continuous functions $\C{(S)}$ uniquely determines $S$ to within homeomorphism. He also remarked on how to obtain the analogous result for the bounded lattice of continuous functions with co-domain the unit interval $[0,1]\subseteq \R$. Kaplansky's result should be compared with the standard Stone-Gelfand-Kolmogorov theorem that the unital commutative ring structure of $\C{(S)}$ determines $S$; see e.g.\ \cite[Thm. 4.9]{GillJer1976}. The set $\C{(S,[0,1])}$ of all $[0,1]$-valued continuous functions on $S$ is naturally an MV-algebra, with operations defined pointwise from the standard MV-algebra $[0,1]$. As an application of   Theorem~\ref{thm:kaplansky},  we prove (Corollary \ref{cor:kaplanskyoriginal}): The underlying lattice  of any separating MV-subalgebra  of $\C{(S,[0,1])}$ uniquely determines the homeomorphism type of $S$.   

By Corollary \ref{cor:rootsystem}, there is a uniquely determined  function
\begin{align*}
m\colon Y \to Z
\end{align*}
that sends each prime MV-ideal to the unique maximal MV-ideal that contains it. This map is continuous with respect to the spectral topology. 
To see this, recall that a distributive lattice $D$ is \emph{normal} \cite[p.\ 67]{Johnstone1982}  if, whenever $d_1, d_2 \in D$ satisfy $d_1\vee d_2 = 1$ (the maximum of $D$), there are $c_1, c_2 \in D$ with $c_1 \wedge c_2 = 0$ (the minimum of $D$) such that $c_1 \vee d_2 = 1$ and $c_2 \vee d_1 =1$. The lattice of open sets of a topological space is normal in this sense precisely when the space is normal in the usual topological sense, whence the terminology. In the Stone dual space $(X(D),\tau^{\downarrow})$ of any distributive lattice $D$, the subspace of maximal points, $\max(X(D))$, forms a $T_{1}$ subspace, because the specialization order on that space is trivial. For normal lattices, however, more can be said.
\begin{lem}[{\cite[Chapt. II, Lem. 3.6]{Johnstone1982}}]\label{lem:johnstone0}Let $D$ be a distributive lattice. If $D$ is normal, then the maximal spectrum $(\max(X(D)),\tau^\downarrow)$ of $D$ is Hausdorff.\end{lem}
\noindent While the converse of Lemma \ref{lem:johnstone0} fails (see \cite[Chapt. II, Example 3.7]{Johnstone1982}), we have:
\begin{lem}[{\cite{Simmons1980}; \cite[Chapt. II, Prop. 3.7]{Johnstone1982}}]\label{lem:johnstone} Let $D$ be a distributive lattice. The following are equivalent.
\begin{enumerate}
\item $D$ is normal.
\item Each prime ideal of $D$ is contained in a unique maximal ideal of $D$.
\item The inclusion map $\max(X(D)) \subseteq X(D)$ admits a continuous retraction $(X(D),\tau^\downarrow)\onto (\max(X(D)),\tau^\downarrow)$. 
\end{enumerate}
\end{lem}
Direct inspection shows that the underlying distributive lattice of an MV-algebra $A$ is in general not normal. 
\begin{exa}\label{ex:notnorm} In the MV-algebra $A:=\C{([0,1],[0,1])}$ of continuous functions from $[0,1]$ to itself, consider the function $d_{1}$ that is constantly equal to $1$ over $[0,\frac{1}{2}]$, takes value $0$ at $1$, and is affine linear over  $[\frac{1}{2},1]$. Further consider the function $d_{2}$ that is 
 constantly equal to $1$ over $[\frac{1}{2},1]$, takes value $0$ at $0$, and is affine linear over  $[0,\frac{1}{2}]$. Then $d_{1}\vee d_{2}=1_{[0,1]}$, where $1_{[0,1]}$ denotes  the function constantly equal to $1$ over $[0,1]$, that is, the top element of the MV-algebra $A$. Suppose now that $c_{1}$ is an element satisfying $c_{1}\vee d_{2}=1_{[0,1]}$. Then $c_{1}$ takes value $1$ at least over the open set $[0,\frac{1}{2})$, because $d_{2}<1$ there. By continuity, $c_{1}$ is $1$ at least over the closure $[0,\frac{1}{2}]$ of $[0,\frac{1}{2})$. Similarly, if $c_{2}\vee d_{1}=1_{[0,1]}$ then $c_{2}$ is $1$ at least over the closed set $[\frac{1}{2},1]$. But then $c_{1}\wedge c_{2}$ is  $1$ at $\frac{1}{2}$, and the underlying lattice of $A$ fails to be normal.\qed
 \end{exa}
However,  the lattice of principal MV-congruences $\mathrm{KCon}\, A$ of any MV-algebra $A$ is normal. This allows us to prove:
\begin{prop}\label{prop:mretract}
The map $m \colon  (Y,\tau^\downarrow) \to (Z,\tau^\downarrow)$  is a retraction of the inclusion map $Z \subseteq Y$, and $(Z,\tau^\downarrow)$ is a compact Hausdorff space.
\end{prop}
\begin{proof}In Lemma \ref{lem:johnstone}, take $D$ to be the lattice $\mathrm{KCon}\, A$ of principal MV-ideals of $A$. Then, by  Proposition \ref{prop:YasKPrin} and Corollary \ref{cor:rootsystem}, $\mathrm{KCon}\, A$  is normal. A further application of Lemma \ref{lem:johnstone} yields the continuity of the retraction $m \colon  (Y,\tau^\downarrow) \to (Z,\tau^\downarrow)$ of $Y$ onto $Z$. Then $(Z,\tau^{\downarrow})$ is compact by  
\cite[Thm. 17.7]{Willard}, because it is the continuous image of the spectral, hence compact, space $(Y,\tau^{\downarrow})$; and it is Hausdorff by Lemma \ref{lem:johnstone0}.
\end{proof}
\begin{rem}\label{rem:completenormality}We observe in passing that the root system property of $(Y, \leq)$ (Corollary~\ref{cor:rootsystem}) corresponds to the fact that $\mathrm{KCon}\,A$ is even \emph{completely normal}; see  \cite[and references therein]{Cignolietal99}.\qed
\end{rem}
The kernel of the function $m \circ k\colon X\to Z$ is easily described from the order of $X$. For this, define a binary relation  $W$ on $X$ by setting
\begin{align}\label{eq:sim}
x_1 W x_2 \text{ iff  there are } x_{1}',x_{2}', x_{0} \in X \text{ such that } x_{1}'\leq x_1,\, x_{2}'\leq x_2,\, x_0\geq x_{1}',x_{2}'.
\end{align}
The  picture below depicts the typical order configuration for which $x_1 W x_2$.

\smallskip
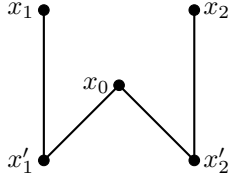
\begin{center}
\begin{tikzpicture}
\po{0,2}
\node[left] at (0,2) {$x_1$};
\po{0,0}
\node[left] at (0,0) {$x_1'$};
\po{1,1}
\node[left] at (1,1) {$x_0$};
\po{2,0}
\node[right] at (2,0) {$x_2'$};
\po{2,2}
\node[right] at (2,2) {$x_2$};
\li{(0,2)--(0,0)--(1,1)--(2,0)--(2,2)}
\end{tikzpicture}
\end{center}
\smallskip

\begin{lem}\label{lem:maximals}
Let $x, x' \in X$ be such that $x \leq x'$. Then $m(k(x)) = m(k(x'))$.
\end{lem}
\begin{proof}
By the Interpolation Lemma~\ref{lem:interpolation}, there is $y := k(x'') \in Y$ such that $y \geq k(x)$ and $y \geq k(x')$. The maximal MV-ideal above $y$ is above both $k(x)$ and $k(x')$. Since, by Corollary \ref{cor:rootsystem},  there is a unique maximal ideal above $k(x)$ and $k(x')$, we get $m(k(x)) = m(k(x'))$.
\end{proof}
\begin{lem}\label{lem:kaplansky1} For any $x_1,x_2\in X$, we have $x_1 W x_2$ if, and only if, $m(k(x_1))=m(k(x_2))$. In other words, $W = \ker{(m\circ k)}$.
\end{lem}
\begin{proof}
Suppose $x_1 W x_2$. Pick $x_1', x_2', x_0 \in X$ as in the definition of $W$. Repeated applications of Lemma~\ref{lem:maximals} yield
\[ m(k(x_1)) = m(k(x_1')) = m(k(x_0)) = m(k(x_2')) = m(k(x_2)).\]
Conversely, suppose $m(k(x_1)) = m(k(x_2))$. Then $x_i \geq k(x_i) =: x_i'$, and $x_0 := m(k(x_1)) = m(k(x_2)) \geq k(x_i)$ for $i = 1,2$. Hence, $x_1 W x_2$.
\end{proof}
Let us write $[x]_{W}$ for the equivalence class of $x \in X$ under $W$, 
 $X/{W}$ for the quotient set, and 
 \[
q \colon X \longrightarrow X/{W}
 \]
 for the natural quotient map $x \mapsto [x]_{W}$. Since $W=\ker(m \circ k)$ by Lemma \ref{lem:kaplansky1}, there is a unique (set-theoretic) function
 $f\colon X/W\to Z$ making the diagram
\begin{equation}\label{eq:f}
\begin{tikzpicture}[baseline=(current  bounding  box.center)]
\matrix (m) [matrix of math nodes, row sep=4em, column sep=4em, text height=1.5ex, text depth=0.25ex] 
{X & X/W \\
  Z  &  \\};

\path[->] (m-1-1) edge node[left] {$m \circ k$} (m-2-1); 
\draw[->,dashed] (m-1-2) edge node[right] {$f$} (m-2-1);
\path[->] (m-1-1) edge node[above] {$q$} (m-1-2);
\end{tikzpicture}
\end{equation}
commute, and this function $f$ is a bijection.
\begin{lem}\label{lem:kaplansky2} The map $m \circ k \colon (X,\tau^{\downarrow}) \to (Z,\tau^{\downarrow})$ is continuous.
\end{lem} 
\begin{proof}
By Proposition~\ref{prop:kcontinuity}, $k \colon (X,\tau^p) \to (Y,\tau^\downarrow)$ is continuous, and by Proposition~\ref{prop:mretract}, $m \colon (Y,\tau^\downarrow) \to (Z,\tau^\downarrow)$ is continuous, so $m \circ k \colon (X,\tau^p) \to (Z,\tau^\downarrow)$ is continuous. Now note that, by Lemma~\ref{lem:maximals}, for any subset $T \subseteq Z$, $(m \circ k)^{-1}(T)$ is both an upset and a downset in $(X,\leq)$. In particular, for any open set $U$ in $(Z,\tau^\downarrow)$, we get that $(m \circ k)^{-1}(U)$ is an open downset in $(X,\tau^p,\leq)$, and therefore it is open in $(X,\tau^\downarrow)$.\footnote{Note that $(m \circ k)^{-1}(U)$ is also open in $\tau^\uparrow$, so that the function $m \circ k$ remains continuous when we put the stronger topology $\tau^\downarrow \cap \tau^\uparrow$ on $X$, but we will not need this in what follows.}
\end{proof}

 \begin{prop}\label{prop:kaplansky}Let $\sigma$ denote the quotient topology that $X/W$ inherits from $(X,\tau^{\downarrow})$. Then the unique function  $f\colon (X/W,\sigma)\to (Z,\tau^{\downarrow})$ making \textup{(\ref{eq:f})} commute is a homeomorphism.
 \end{prop}
\begin{proof}
Note that $f$ is continuous by Lemma~\ref{lem:kaplansky2}, the fact that (\ref{eq:f}) commutes, and the definition of the quotient topology.
Since $q \colon (X,\tau^\downarrow) \to (X/W,\sigma)$ is continuous and onto by definition,  by \cite[Thm. 17.7]{Willard} the space $(X/W,\sigma)$ is compact because $(X,\tau^\downarrow)$ is. Recall that $(Z,\tau^\downarrow)$ is a Hausdorff space. So $f$ is a continuous bijection from a compact space to a Hausdorff space, and therefore $f$ is a homeomorphism by \cite[Thm. 17.14]{Willard}.
\end{proof}

We are  ready to establish the MV-algebraic generalization of Kaplansky's Theorem.
\begin{thm}[MV-algebraic Kaplansky's Theorem]\label{thm:kaplansky}MV-algebras with isomorphic underlying lattices have homeomorphic maximal MV-spectra, i.e.\ collections of maximal MV-ideals equipped with the spectral topology $\tau^{\downarrow}$.
\end{thm}
\begin{proof} It suffices to show that we can recover $(Z,\tau^{\downarrow})$ from $(X,\tau^{\downarrow})$ only, up to homeomorphism. This is precisely the content of Proposition \ref{prop:kaplansky}, upon noting that (\ref{eq:sim}) defines the relation $W$ in terms of the specialization order of $(X,\tau^{\downarrow})$.
\end{proof}
For any compact Hausdorff space $S$, let $\C{(S, [0,1])}$ denote the MV-algebra of all continuous $[0,1]$-valued functions on $S$, where $[0,1]$ is endowed with its standard (Euclidean) topology, and the  operations are defined pointwise from the standard MV-algebra $[0,1]$. A subset $B\subseteq \C{(S, [0,1])}$ is said to \emph{separate the points of $S$} if for each $p\neq q \in S$ there is $b \in B$ with $b(p)=0$, $b(q)>0$. We can now obtain (a stronger version of)  \cite[Thm. 1 and Thm. 6]{kaplansky}.
\begin{cor}[Kaplansky's Theorem]\label{cor:kaplanskyoriginal} Let  $S_{1}, S_{2}$ be any two compact Hausdorff spaces. \textup{(i)} \, Suppose $A_{i}\subseteq  \C{(S_{i},[0,1])}$ ($i = 1,2$) are MV-subalgebras, and that each $A_i$ separates the points of $S_{i}$. If $A_{1}$ and $A_{2}$ are isomorphic as lattices, then $S_{1}$ and $S_{2}$ are homeomorphic. \textup{(ii)} \,
The lattices $\C{(S_{1},[0,1])}$ and $\C{(S_2, [0,1])}$ are isomorphic if, and only if, the spaces $S_{1}$ and $S_{2}$ are homeomorphic.
\end{cor}
\begin{proof} (i) \, Let $Z_{i}$  denote the maximal MV-spectrum of $A_{i}$,  $i=1,2$, equipped with the spectral topology $\tau^{\downarrow}$. Since $A_{i}$ separates the points of $S_{i}$, by \cite[Thm. 3.4.3]{cdm} it follows that $S_i$ is homeomorphic to $Z_{i}$, $i = 1,2$. By assumption, the lattice reducts of $A_1$ and $A_2$ are isomorphic. By Theorem~\ref{thm:kaplansky}, $Z_{1}$ and $Z_{2}$ are homeomorphic. So $S_1$ is homeomorphic to $S_2$.

(ii)\,
For the non-trivial implication, write $A_i$ for the MV-algebra $\C{(S_i,[0,1])}$, $i = 1,2$. Now $S_i$ is compact Hausdorff by hypothesis, and therefore by Urysohn's Lemma the elements of $A_i$ separate the points of $S_i$.  By \cite[Thm. 3.4.3]{cdm} it follows that $S_i$ is homeomorphic to the maximal MV-spectrum of $A_i$ ($i = 1,2$) with the spectral topology. Now apply (i).
\end{proof}
Theorem \ref{thm:kaplansky} does not extend without substantial modifications to prime MV-spectra, as the following example shows. 
\begin{exa}\label{ex:kap}(For readers familiar with the $\Gamma$ functor and lattice-groups \cite{Mun1986}, the MV-algebra we consider here is the unit interval $\Gamma(\,\Q \vec{\times} \Q,(1,1)\,)$, where $\Q\vec{\times}\Q$ denotes, as usual, the lexicographic product of the ordered additive abelian group $\Q$ of rational numbers with itself.)  Totally order the set-theoretic Cartesian product $\Q\times \Q$ lexicographically. 
That is, define $(a,b)\preceq (a',b')$, for $a,b,a',b' \in \Q$, to mean  $a< a'$, or else $a=a'$, and then $b\leq b'$. Consider the interval $Q':=\{x\in \Q\times \Q\mid  (0,0)\preceq x \preceq (1,1)\}$.
Define the operation $\oplus$ on $Q'$ by setting $(a,b)\oplus(a',b'):=\min{\{(a+a',b+b'),(1,1)\}}$, where the minimum is taken with respect to $\preceq$. Further define a unary operation $\neg$ on $Q'$ by $\neg (a,b):= (1-a,1-b)$. It is straightforward to check that $Q'$ is an MV-algebra under the binary operation $\oplus$, the unary operation $\neg$, and the constant $(0,0)$. A further verification shows that the underlying order of this MV-algebra coincides with the restriction of $\preceq$ to $Q'$. Hence $Q'$ is a totally ordered MV-algebra whose underlying lattice is a dense countable chain with endpoints. And any two such chains are order-isomorphic \cite[Cor. 2.9]{rosenstein}. The subalgebra $Q=[0,1]\cap \Q$ of the standard MV-algebra $[0,1]$ has as underlying lattice a dense countable chain with endpoints, too. Then the MV-algebras $Q'$ and $Q$ have isomorphic underlying lattices. Now direct inspection shows that $Q$ is non-trivial and simple (=it has no non-trivial congruences) and therefore its prime MV-spectrum is the singleton $\{0\}$; whereas $Q'$ has a unique non-maximal prime MV-ideal --- namely, the MV-ideal $\{(0,0)\}$, which lies below the unique maximal MV-ideal $\{(0,b) \mid b \in \Q, b \geq 0\}$ --- whence its MV-spectrum is a doubleton.  \qed
\end{exa}
\section{Sheaf representations of MV-algebras}\label{sec:spectralproduct}
We  now combine the theory  developed in Section~\ref{sec:sheaf} to the specific information about dual spaces of MV-algebras obtained in Section~\ref{sec:decomp}. This leads to the two known sheaf representations of MV-algebras discussed in the Introduction; one over the space of prime MV-ideals, the other over the space of maximal MV-ideals. 
\subsection{Sheaves of MV-algebras indexed over the prime MV-spectrum}\label{subsec:sheaves_DP}
The space $(Y,\tau^\downarrow)$ of prime MV-ideals of the MV-algebra $A$ with the spectral topology is a spectral space, and hence in particular a stably compact space (Example~\ref{exa:stabcompsp}). The map $k : X \to Y$ defined in Definition~\ref{def:k} is continuous from $(X,\tau^p)$ to $(Y,\tau^\downarrow)$ (Proposition~\ref{prop:kcontinuity}), and surjective, as readily follows from Proposition~\ref{prop:fixpointsk}. Thus, the following lemma suffices to obtain a sheaf representation from the map $k$.

\begin{lem}\label{lem:PforDuPo}
The function $k \colon (X,\tau^p,\leq) \to (Y,\tau^\downarrow)$ satisfies property \textup{(P)} from Theorem~\ref{thm:decompositiongivessheaf}.
\end{lem}
\begin{proof}
Let $(U_l)_{l=1}^n$ be a cover of $Y$ by open sets from the topology $(\tau^\downarrow)^\dual = \tau^\uparrow$, and let $(K_l)_{l=1}^n$ be a finite collection of clopen downsets of $X$ which satisfies, for any $l,m \in \{1,\dots,n\}$: $K_l \cap k^{-1}(U_l \cap U_m) = K_m \cap k^{-1}(U_l \cap U_m)$. We need to show that the set 
\begin{align}\label{eq:K}
K := \bigcup_{l=1}^n (K_l \cap k^{-1}(U_l))
\end{align}
is a clopen downset in $X$. 

Since the compact-opens form a basis for $(Y,\tau^\uparrow)$, we may assume that each $U_l$ is compact-open in $\tau^\uparrow$, by passing to a finer cover, which we may still assume to be finite since $Y$ is compact. By the Stone-Priestley Duality Theorem, there are $u_1,\dots,u_n \in A$ such that $U_l = \hat{u_l}^{\,c}$ for each $l \in \{1,\dots,n\}$. Hence, by Proposition~\ref{prop:kcontinuity}, $k^{-1}(U_l)$ is closed for each $l$. Since each $K_l$ is closed, it now follows that $K$ is closed. 

To show that $K$ is open, we  \emph{claim} that $K^c = \bigcup_{l=1}^n (K_l^{\,c} \cap k^{-1}(U_l))$, which is clearly closed. To prove this claim, first suppose $x \in K^c$. Since the $U_l$'s cover $Y$, pick $l$ such that $k(x) \in U_l$. Since $x \not\in K$, we have $x \not\in K_l$, so $x \in K_l^{\,c}\, \cap\, k^{-1}(U_l)$. For the converse inclusion, suppose $x \in K_m^{\, c}\, \cap \, k^{-1}(U_m)$ for some $m$. To show that $x \in K^c$, we need to show that $x \not\in K_l \cap k^{-1}(U_l)$ for each $l$. If $x \in k^{-1}(U_l)$, we have $x \in k^{-1}(U_l \cap U_m)$. By the assumption on the $K_l$'s and the fact that $x \not\in K_m$, we get that $x \not\in K_l$. This proves the \emph{claim}, so $K$ is clopen.

Finally, we show that $K$ is a downset. Suppose $x' \in K$ and $x \leq x'$. We prove that $x \in K$. Since $x' \in K$, pick $l \in \{1, \dots, n\}$ such that $x' \in K_l \cap k^{-1}(U_l)$ and since the $U_l$'s cover $Y$, pick $m \in \{1,\dots,n\}$ such that $k(x) \in U_m$. By the Interpolation Lemma~\ref{lem:interpolation}, pick $x'' \in X$ such that $x \leq x'' \leq x'$, $k(x'') \geq k(x)$ and $k(x'') \geq k(x')$. Then $x'' \in K_l$ since $K_l$ is a downset. Also, since $k(x') \in U_l$ by the choice of $l$, we get $k(x'') \in U_l$ because $U_l$ is an upset, and similarly, we get  $k(x'') \in U_m$ because $k(x)$ is in the upset $U_m$. Hence, $x'' \in K_l \cap k^{-1}(U_l \cap U_m) = K_m \cap k^{-1}(U_l \cap U_m)$ by the assumption on the $K_l$'s. Since $K_m$ is a downset we get that $x \in K_m$, so $x \in K_m \cap k^{-1}(U_m) \subseteq K$.
\end{proof}

To obtain the first sheaf representation from this result, we reason as follows. Recall from Section~\ref{sec:sheaf} that the underlying set of the \'etale space $E$ is defined to be the disjoint union of the algebras $A_y$, for $y \in Y$. Here, $A_y$ is defined to be the quotient of $A$ corresponding to the subspace $k^{-1}({\uparrow}y)$. Now recall from Proposition~\ref{prop:chainusingk} that $k^{-1}({\uparrow}y) = C_y$, and by Proposition~\ref{prop:quotientdually}, the algebra corresponding to the subspace $C_y$ is $A/I_y$. Hence, $A_y = A/I_y$, the map $A \onto A_y$ is an MV-algebra homomorphism, and $E$ is the disjoint union of the MV-algebras $A_y$, for $y \in Y$, with the topology given by the global sections $s_a$. This is  exactly the \'etale space  used in \cite[Chapt. 5]{Yang06} and\cite[Sec. 2.1]{DuPo2010}, so that we obtain, by applying Corollary~\ref{cor:dlesheaf} in light of Lemma \ref{lem:PforDuPo}:
\begin{cor}[{\cite[Prop. 5.1.2 and Rem. 5.3.11]{Yang06}} and {\cite[Thm. 3.12]{DuPo2010}}]\label{cor:DuPo}
Any MV-algebra is isomorphic to the MV-algebra of global sections of the \'etale space over $(Y, \tau^\uparrow)$  whose stalk at $y \in Y$ is the quotient by the prime MV-ideal $I_y$.\qed
\end{cor}
\subsection{The Chinese Remainder Theorem}\label{ss:CRT}
We digress to compare our Lemma~\ref{lem:PforDuPo} to  results in the literature.
The following purely algebraic statement easily follows from Lemma~\ref{lem:PforDuPo} using duality.
\begin{cor}[Chinese Remainder Theorem for principal ideals]\label{cor:CRT}
Let $I_1, \dots, I_n$ be finitely many  principal MV-ideals of the MV-algebra $A$ satisfying $\bigcap_{l=1}^{n} I_{l}=\{0\}$, and let $a_1, \dots, a_n \in A$ be such that
\[ a_i \equiv a_j \mod I_i \oplus I_j\]
for all $i,j = 1, \dots, n$. Then there exists a unique $b \in A$ such that
\[ b \equiv a_l \mod I_l\]
for each $l = 1, \dots, n$.
\end{cor}
\begin{proof}[Sketch of Proof]
Principal MV-ideals correspond to compact open subsets of $(Y,\tau^\uparrow)$. Hence, $(I_{l})_{l=1}^{n}$ yields a family $(U_{l})_{l=1}^{n}$ of compact open subsets. Since $\bigcap_{l=1}^{n}I=\{0\}$, the family  $(U_{l})_{l=1}^{n}$ covers $Y$. The condition on  $(a_l)_{l=1}^{n}$ says that their corresponding clopen downsets $(K_{l})_{l=1}^{n}$ in the Priestley space $(X,\tau^{p},\leq)$ satisfy  $(\ref{label:star})$ in property (P) of Theorem \ref{thm:decompositiongivessheaf}, using the map $k$ in place of $q$. Lemma~\ref{lem:PforDuPo} now yields the desired $b$ via its dual clopen downset (\ref{eq:K}).\end{proof}

In \cite[Thm. 2.6]{FerrLett2011} (cf.\  also  \cite[Lem. 1]{DuPo2012}), the authors prove a general Chinese Remainder Theorem for MV-algebras, where the restriction to principal ideals is not necessary. In the case of principal ideals, the proof  exhibits a specific MV-algebraic expression for   $b$  in the elements  $(a_l)_{l=1}^n$ and $(u_l)_{l=1}^n$, whenever the $u_l$'s are elements of $A$ that generate the principal MV-ideals $(I_{l})_{l=1}^{n}$, i.e.\  satisfy  $U_l = \hat{u}_l^{\, c}$ for each $l=1,\ldots,n$. We will obtain such a stronger version in Subsection \ref{subsec:termdef}.
 We further observe that  Dubuc and Poveda's Pullback-Pushout Lemma \cite[Lem. 3.11]{DuPo2010},     a key ingredient in their proof of the  sheaf representation, is an easy consequence of the Chinese Remainder Theorem. In \cite{FilGeo1995}, the Chinese Remainder Theorem does not feature explicitly, though its implicit r\^{o}le is clear e.g.\ in \cite[Prop. 2.16]{FilGeo1995}.
 
In the literature on lattice-groups, already Keimel  proved a Chinese Remainder Theorem \cite[Lem.  10.6.3]{BKW1977}  in developing his sheaf representation; and Yang  \cite[Prop.  5.1.2]{Yang06} applies it  to establish the sheaf representation via the co-compact dual of the spectral topology. It appears that Keimel's standard result, whose proof admits a straightforward translation in the language of MV-algebras, went unnoticed in much of the MV-algebraic literature. Indeed, \cite{FilGeo1995}, \cite{DuPo2010} and \cite{DuPo2012} do not refer to it. 

Finally, let us point out that sheaf representations and Chinese Remainder Theorems were  studied at the level of universal algebra by Vaggione in \cite{Vag1992},  whose results extend  the previous treatment by Krauss and Clark \cite{KraCla1979}; see also the earlier papers \cite{Wol1974, Cor1977} in the same direction.
\subsection{Term-definability in property (P)}\label{subsec:termdef}
In the proof of the property (P) in Lemma~\ref{lem:PforDuPo}, we showed that the set $K$ in (\ref{eq:K}) is a clopen downset. Therefore, it must be of the form $\hat{b}$ for a unique $b \in A$ by Priestley duality (Corollary~\ref{cor:prdualobjects}).   We now exhibit an explicit MV-algebraic term for $b$  using  one additional lemma concerning the map $k$, whose proof is an application of canonical extensions.
\begin{lem}\label{lem:countableintersection}
For any $a, u \in A$, we have 
\[\hat{a} \cap k^{-1}(\hat{u}^{\, c}) \subseteq \bigcap_{m = 0}^\infty \widehat{a \ominus mu} \subseteq \hat{a} \cap {\downarrow}k^{-1}(\hat{u}^{\, c}).\]
\end{lem}
\begin{proof}
Suppose that $x \in \hat{a} \cap k^{-1}(\hat{u}^{\, c})$. We show by induction on $m$ that $x \in \widehat{a \ominus mu}$ for all $m \geq 0$. For $m = 0$, we have $x \in \hat{a}$ by assumption. Suppose that $x \in \widehat{a \ominus mu}$ for some $m$. We show that $x \in \widehat{a \ominus (m+1)u}$.  Since $x \in k^{-1}(\hat{u}^c)$, we have $u \in I_{k(x)}$. Therefore, by Proposition~\ref{prop:kcontinuity} and the assumption that $a \ominus mu \not\in I_x$, we get $(a \ominus mu) \ominus u \not\in I_x$. Now, since $(a \ominus mu) \ominus u \leq a \ominus (m+1)u$, we also have $a \ominus (m+1)u \not\in I_x$, as required.

For the second inclusion, suppose that $x \in \bigcap_{m = 0}^\infty \hat{a \ominus mu}$. Putting $m = 0$, it is clear that $x \in \hat{a}$. To show that $x \in {\downarrow}k^{-1}(\hat{u}^{\, c})$, let
\[ x' := x \iplus \left(\bigvee_{m \geq 0} mu\right).\]
By Proposition~\ref{prop:primeplusidealcanext}.1, we have either $x' \in M^\infty$ or $x' = 1$. If $x' = 1$, then in particular $a \leq x'$, but compactness then yields $m \geq 0$ such that $a \leq x \iplus mu$, i.e.\ $a \ominus mu \leq x$, which is impossible by assumption. We therefore conclude $x' \in M^\infty$, and clearly $x \leq x'$. To show that $x' \in k^{-1}(\hat{u}^{\, c})$, note that
\[ x' \iplus u = x \iplus \left(\bigvee_{m \geq 0} mu\right) \iplus u = x \iplus \left(\bigvee_{m \geq 0} (m+1)u\right) \leq x',\]
so $u \leq k(x')$, by definition of $k(x')$.
\end{proof}
In the next corollary we write $[a]_y$ to denote the congruence class of $a\in A$ modulo $I_y$ in the quotient MV-algebra $A/I_y$. 
\begin{cor}[Term-definability in property (P)]
Let $(u_i)_{i=1}^n$ and let $(a_i)_{i=1}^n$ be two finite collections of elements of $A$ such that $(\hat{u}_i^{\, c})_{i=1}^n$ covers $Y$, and the $a_i$ are compatible with this cover, i.e.,
\[ \text{For all } y \in Y, \text{ if } y \in \hat{u}_i^{\, c} \cap \hat{u}_j^{\, c}, \text{ then } [a_i]_y = [a_j]_y.\]
There exists an integer $t \geq 0$ such that, setting
\[
b := \bigvee_{i=1}^n (a_i \ominus t u_i)\,,
\]
we have $[b]_y = [a_i]_y$ for all $y \in \hat{u}_i^{\, c} \cap Y$.
\end{cor}
\begin{proof}
By Lemma~\ref{lem:PforDuPo}, $K := \bigcup_{i=1}^n (\hat{a}_i \cap k^{-1}(\hat{u}_i^{\, c}))$ is a clopen downset. Fix $i \in \{1,\dots,n\}$. Then $\hat{a}_i \cap {\downarrow} k^{-1}(\hat{u}_i^{\, c}) \subseteq K$, so $\bigcap_{m \in \N} \widehat{a_i \ominus mu_i} \subseteq K$ by Lemma~\ref{lem:countableintersection}. Note that, for $m \leq m'$, we have $a_i \ominus m' u_i \leq a_i \ominus m u_i$ using Lemma \ref{lem:minusplusadjunction}. Therefore, $(\hat{a_i \ominus mu_i})_{m \in \N}$ is a decreasing chain of closed sets in $\tau^p$. Since $K$ is compact and it contains the full intersection of this chain, there exists $t_i \geq 0$ such that $\hat{a_i \ominus t_iu_i} \subseteq K$. Choosing such $t_i$ for each $i \in \{1,\dots,n\}$ and setting $t := \max{\{ t_i \mid  i = 1,\dots,n\}}$, we now have
\[ K \subseteq \bigcup_{i=1}^n \left(\bigcap_{m \in \N} \widehat{a_i \ominus mu_i}\right) \subseteq \bigcup_{i=1}^n \left(\hat{a_i \ominus t u_i}\right) \subseteq \bigcup_{i=1}^n \left(\hat{a_i \ominus t_iu_i}\right) \subseteq K,\]
so $K = \bigcup_{i=1}^n \left(\hat{a_i \ominus t u_i}\right)$. Now, putting $b = \bigvee_{i=1}^n (a_i \ominus t u_i)$, we get $\hat{b} = K$, so that this $b$ satisfies the required property, by Priestley duality and Corollary~\ref{cor:DuPo}.
\end{proof}

\subsection{Sheaves of MV-algebras indexed over the maximal MV-spectrum}\label{subsec:sheaves_FG}
The space $(Z,\tau^\downarrow)$ of maximal MV-ideals of $A$ with the spectral topology is a compact Hausdorff space (Proposition \ref{prop:mretract}), and hence in particular a stably compact space (Example~\ref{exa:stabcompsp}.1). Since $(Z,\tau^\downarrow)$ is Hausdorff, the topology $\tau^\downarrow$ restricted to $Z$ is equal to its own co-compact dual topology, by Theorem~\ref{thm:stabcompsp}.5. The composite map $m \circ k : X \to Z$ is continuous from $(X,\tau^p)$ to $(Z,\tau^\downarrow)$ by Lemma~\ref{lem:kaplansky2}, and surjective, as readily follows from Proposition~\ref{prop:fixpointsk} and Proposition~\ref{prop:mretract}. In analogy with Lemma \ref{lem:PforDuPo}:

\begin{lem}\label{lem:mksatisfiesP}
The function $m \circ k : (X,\tau^p,\leq) \to (Z,\tau^\downarrow)$ satisfies property \textup{(P)} from Theorem~\ref{thm:decompositiongivessheaf}.
\end{lem}
\begin{proof}
Let $(U_l)_{l=1}^n$ be a cover of $Z$ by open sets in the topology $(\tau^\downarrow)^\dual = \tau^\downarrow$, and let $(K_l)_{l=1}^n$ be a finite collection of clopen downsets of $X$, which satisfies the hypothesis of property (P). We need to show that $K := \bigcup_{l=1}^n (K_l \cap (m \circ k)^{-1}(U_l))$ is a clopen downset in $X$. Since $m\circ k\colon (X,\tau^\downarrow)\to (Z,\tau^\downarrow)$ is continuous (Lemma~\ref{lem:kaplansky2}), each set $(m \circ k)^{-1}(U_l)$ is an open downset, so $K$ is an open downset. To see that $K$ is moreover closed, notice that $K^c = \bigcup_{l=1}^n (K_l^{\,c} \cap (m \circ k)^{-1}(U_l))$: the proof of this equality is, \textit{mutatis mutandis}, the same as the proof of the claim in Lemma~\ref{lem:PforDuPo}. Now $K^c$ is clearly open, so $K$ is closed.
\end{proof}

Again, we  obtain a sheaf representation $p \colon E \to Z$ from the map $m \circ k$.  Recall from Section~\ref{sec:sheaf} that the underlying set of the \'etale space $E$ is defined to be the disjoint union of algebras $A_z$, for $z \in Z$. Here, $A_z$ is defined to be the quotient of $A$ corresponding to the subspace $(m \circ k)^{-1}({\uparrow}z) = (m \circ k)^{-1}(z)$. The MV-algebra corresponding to the subspace $(m \circ k)^{-1}(z)$ can be obtained as follows.
\begin{prop}\label{prop:germinal}For any $z\in Z$, the closed subspace $(m \circ k)^{-1}(z)$ of $X$ is homeomorphic to the Priestley dual space of the lattice underlying the MV-algebraic quotient $A/\mathfrak{o}_{z}$ of $A$, where the MV-ideal $\mathfrak{o}_{z}$ is defined by\footnote{That is, the MV-ideal $\mathfrak{o}_{z}$ is the intersection of all prime MV-ideals contained in the maximal MV-ideal $I_z$. This is  the \emph{germinal MV-ideal at $z \in Z$}, see \cite[Def. 4.7]{Mun2011}.} 
\[
\mathfrak{o}_{z}:=\bigcap_{y \,\in \, {\downarrow} z \cap Y}I_y.
\]
\end{prop}
\begin{proof}
By Proposition~\ref{prop:quotientdually}, the quotient $A/\mathfrak{o}_{z}$ corresponds under Priestley duality to the subspace
\[
S_{\mathfrak{o}_{z}}=\{x\in X \mid I_{x}\overline{\oplus}\,\mathfrak{o}_{z} \subseteq I_{x}\}\,.
\]
Recall that, for any MV-ideal $J$, we have $I_x \oplus J \subseteq I_x$ if, and only if, $J \subseteq I_{k(x)}$, so we can write
\[
S_{\mathfrak{o}_{z}}=\{x\in X \mid \mathfrak{o}_{z} \subseteq I_{k(x)}\}\,.
\]
We will now prove that $S_{\mathfrak{o}_{z}} = (m\circ k)^{-1}(z)$. Let $x \in X$, and suppose first that $m\circ k(x) = z$. Then $I_{k(x)} \subseteq I_z$ by definition of $m$, so $\mathfrak{o}_z \subseteq I_{k(x)}$ by definition of $\mathfrak{o}_z$. 

Conversely, suppose that $m\circ k(x) \neq z$. Then $k(x) \nleq z$, since every prime MV-ideal is contained in a unique maximal ideal (Corollary \ref{cor:rootsystem}), and also that $z \nleq k(x)$, since $z$ is a maximal MV-ideal. By Proposition~\ref{prop:separation}, pick $u, v \in A$ such that $k(x) \in \hat{u}$, $z \in \hat{v}$ and $u \wedge v = 0$. If $I_y$ is any prime MV-ideal which is contained in $z$, then in particular $0 \in I_y$ and $v \not\in I_y$, so that $u \in I_y$. Hence, $u \in \mathfrak{o}_z$ by definition of $\mathfrak{o}_z$. However, by construction, we have $u \not\in I_{k(x)}$, so $\mathfrak{o}_z \not\subseteq I_{k(x)}$, and therefore $x \not\in S_{\mathfrak{o}_{z}}$.
\end{proof}

By the preceding proposition, the stalk of the \'etale space $E$ at $z \in Z$ is the MV-algebra $A/\mathfrak{o}_z$. So $E$ is the disjoint union of the MV-algebras $A/\mathfrak{o}_z$, for $z \in Z$, with the topology given by the global sections $s_a$. This  is exactly the \'etale space used in  \cite{FilGeo1995}, so that we obtain, by applying Corollary~\ref{cor:dlesheaf} in the light of Lemma~\ref{lem:mksatisfiesP}:
\begin{cor}[{\cite[Prop. 2.16]{FilGeo1995}}]\label{cor:FilGeo}
Any MV-algebra is isomorphic to the MV-algebra of global sections of the \'etale space over $(Z,\tau^\downarrow)$ whose stalk at $z \in Z$ is the quotient by the MV-ideal $\mathfrak{o}_z$ at $z$.\qed
\end{cor}
\begin{rem}\label{rem:order comp}As shown in Section~\ref{sec:kaplansky}, the sets appearing in the partition  $\{(m\circ k)^{-1}(z)\}_{z \in Z}$ are the order-components of $(X,\leq)$, and are in fact also the equivalence classes of the equivalence relation $W$ defined there. Moreover, every $W$-equivalence class contains exactly one point of $Z$, and the space $X/W$ is homeomorphic to $Z$, as was  proved in Proposition~\ref{prop:kaplansky}.\qed
\end{rem}
\section{Further research}\label{sec:further}
The method of Theorem~\ref{thm:decompositiongivessheaf} can be applied to any variety of algebras with a distributive lattice reduct. It would be of interest to investigate other significant classes of algebras, besides lattice-ordered abelian groups and MV-algebras, for which sheaf-theoretic representations may be obtained from an \'etale decomposition of the lattice spectrum. We also remark that condition (P) in Theorem \ref{thm:decompositiongivessheaf} is sufficient but far from necessary in order  that the map $\eta$ be a surjection. It would be  worthwhile to explore  weaker sufficient conditions.

Our Proposition \ref{prop:plusprops} determines enough of the structure of the dual Stone-Priestley space of an MV-algebra to establish our main results on \'etale decompositions.  It has been a long-standing open problem in the theory of lattice-ordered abelian groups to characterise spectra of prime congruences in topological terms, see \cite[and references therein]{Cignolietal99}. The problem of characterizing the dual Stone-Priestley space of a unital lattice-ordered abelian group is likewise open, to the best of our knowledge. In light of Proposition \ref{prop:plusprops}, we wonder whether a  characterisation that uses the partial semigroup structure described there is possible. A further question of interest is to investigate the functoriality of the construction under appropriate duals of MV-algebraic homomorphisms.

Theorem \ref{thm:kaplansky}, the MV-algebraic generalization  of Kaplansky's theorem, shows that one can construct the maximal MV-spectrum of any MV-algebra from its underlying lattice only. Example \ref{ex:kap} shows that this fails for prime MV-spectra. The question arises  whether there are significant classes of MV-algebras with the property that the prime MV-spectrum is uniquely determined by the underlying lattice of the MV-algebra.  Recall that an algebra in a variety is called \emph{simple} if it has no non-trivial congruences, and \emph{semisimple} if it is a subdirect product of simple algebras.  Semisimple MV-algebras are precisely those that arise as subalgebras of the MV-algebra $\C{(S, [0,1])}$ of all continuous $[0,1]$-valued functions on a compact Hausdorff space $S$, see \cite[Cor. 3.6.8]{cdm}. Example \ref{ex:kap} uses a non-semisimple MV-algebra. An MV-algebra that is semisimple, and is such that each of its quotient algebras also has that property, is known as \emph{hyperarchimedean} in the MV-algebraic literature \cite[Sec. 6.3]{cdm}. Equivalently, an MV-algebra is hyperarchimedean just in case it has no non-maximal prime MV-ideals \cite[Sec. 6.3.2]{cdm}. Therefore the underlying lattice of a hyperarchimedean MV-algebra determines its prime MV-spectrum to within a homeomorphism by Theorem \ref{thm:kaplansky}.  We wonder whether the semisimple property, or a  strengthening thereof that is weaker than  the hyperarchimedean property, suffices to guarantee that the prime MV-spectrum of an MV-algebra be uniquely determined by its underlying lattice.

\section*{Acknowledgements}
The work of the first-named author was partially supported by ANR 2010 BLAN 0202 02 FREC. The PhD research project of the second-named author has been made possible by  grant 617.023.815 of the Netherlands Organization for Scientific Research (NWO). The work of the third-named author was partially supported by the Italian FIRB ``Futuro in Ricerca'' grant RBFR10DGUA, and by the Italian PRIN grant 2008JKBJJF.
The authors also gratefully acknowledge the support of these last two grants for funding a research visit to Milan of the first-named and second-named authors.
We are grateful to Roberto Cignoli,  Lauren{\c{t}}iu Leu{\c{s}}tean, and Daniele Mundici for bringing \cite{Vag1992},   \cite{DiNolaLeustean2003}, and \cite{Schwartz2013} to our attention, respectively. We also thank Klaus Keimel for his valuable comments. Finally, we thank the anonymous referees for their valuable comments and suggestions.

\bibliographystyle{amsplain}
\bibliography{GGM}

\end{document}